\documentclass[leqno]{amsart}
\usepackage{amsmath}
\usepackage{amssymb}
\usepackage{hyperref}
\usepackage{color}
\allowdisplaybreaks

\makeatletter
\@namedef{subjclassname@2020}{Subject}
\makeatother

\newcommand\bQ{\ensuremath{\mathbb{Q}}}
\newcommand\bF{\ensuremath{\mathbb{F}}}
\newcommand\bP{\ensuremath{\mathbb{P}}}
\newcommand\bZ{\ensuremath{\mathbb{Z}}}
\newcommand\bC{\ensuremath{\mathbb{C}}}
\newcommand\cF{\ensuremath{\mathcal{F}}}

\newcommand\cL{\ensuremath{\mathcal{L}}}
\newcommand\cM{\ensuremath{\mathcal{M}}}
\newcommand\cN{\ensuremath{\mathcal{N}}}
\newcommand\cO{\ensuremath{\mathcal{O}}}
\newcommand\cR{\ensuremath{\mathcal{R}}}

\newcommand\sD{\ensuremath{\mathsf{D}}}

\newcommand\diff[1]{\ensuremath{\mathrm{d}#1}}

\DeclareMathOperator\tors{tors}

\DeclareMathOperator\Card{\sharp}
\newcommand\case[1]{{\ensuremath{{\mathrm(}{\mathit{#1\/}}{\mathrm)}}}}

\theoremstyle{plain}
\newtheorem{theorem}{Theorem}[section]
\newtheorem{lemma}[theorem]{Lemma}
\newtheorem{prop}[theorem]{Proposition}
\newtheorem{cor}[theorem]{Corollary}

\theoremstyle{remark}
\newtheorem{rem}[theorem]{Remark}
\newtheorem{rems}[theorem]{Remarks}

\newtheorem{notation}[theorem]{Notation}

\title[Bounds on points of given order]{Bounds on the number of torsion points of given order on curves embedded in their jacobians}
\author{John Boxall}
\address{Laboratoire de Math\'ematiques Nicolas Oresme, UMR CNRS 6139, Campus 2, Universit\'e de Caen-Normandie, 14032 Caen cedex, France}
\email{john.boxall (at) wanadoo.fr}
\keywords{Algebraic curves, Jacobians, Torsion points}
\subjclass[2020]{Primary 14H05, 14H40, 11G10; Secondary 14H55, 11G20}

\date{\today}
\begin{document}

\begin{abstract}
Working over an algebraically closed field of arbitrary characteristic we study, for integers $N\geq 2$ and $g\geq 2$, the set of points of order dividing $N$ lying on an irreducible smooth proper curve of genus $g$ embedded in its jacobian using a fixed base point.  We discuss bounds on its cardinality and describe an efficient method for computing the set. Our method uses wronskians similar to those used in the study of Weierstrass points and the strength of our bounds is related to whether or not a certain multiple of the curve is contained in the negative of the theta divisor. Several examples are discussed.  This generalizes our previous work dealing with the case of hyperelliptic curves embedded in their jacobian using a Weierstrass point as base point. 
\end{abstract}

\maketitle

\section{Introduction}

Let $k$ be an algebraically closed field of characteristic $p\geq 0$ and let $X$ be an irreducible smooth proper curve over $k$ of genus $g\geq 2$. Fix a point $\infty$ on $X(k)$ and denote by $J$ the jacobian variety of $X$; recall that $J$ is an abelian variety of dimension $g$ defined over $k$. We always think of $X$ as embedded in $J$ using $\infty$ as base point. For any integer $N\geq 1$, denote by $J[N]$ the group of points of $J(k)$ of order dividing $N$ and write $X[N]$ for $J[N]\cap X$.  If $p>0$, we say that $p$ is \emph{purely inseparable for $X$} if the multiplication-by-$p$ isogeny of $J$ is purely inseparable. The following was proved in \cite{Bo23}, using a method based on an idea of Coleman, Kaskel and Ribet \cite{CoKaRi99}. 
\begin{prop} \label{prop:oldbound} We have $\Card{X[2]}\leq 2g+2$ and, if $N\geq 3$, then
\begin{equation*}
\Card{X[N]}\leq g(N-1)^2
\end{equation*}
except when $p$ is purely inseparable for $X$ and $N-1$ is a power of $p$, in which case
\begin{equation*}
\Card{X[N]}\leq g(N+1)^2
\end{equation*}
unless $(p,N)=(2,3)$.
\end{prop}

This bound may be viewed as cubic in the variables $g$ and $N$ and one can ask whether it can be improved to a quadratic bound, in other words whether there exists a polynomial $Q[X,Y]\in \bZ[X,Y]$ of total degree $2$ such that $\Card{X[N]}\leq  Q(g,N)$ for all curves $X$ of genus $g$ (always assumed irreducible smooth and proper) and all $N\geq 1$. We are unable to answer this question in general, and have to content ourselves with the following result, which we state for $N\geq 2g-1$ for simplicity. (See Proposition~\ref{prop:genintromain} for a generalization to all $N\geq 2$.) If $r\in \{0,1,\dots, g-1\}$, we write $W_r$ be the subvariety $\{\sum_{i=1}^r\xi_i\mid \xi_i\in X\}$ of $J$ and $W^-_r$ for the image of $W_r$ by multiplication by $-1$ in $J$.

\begin{theorem} \label{theo:intromain}  Let $X$, $g$, $N$ be as above and suppose $N\geq 2g-1$. Then \emph{either}
\begin{equation*}
(N-g+1)_*X\subseteq W^-_{g-1}
\end{equation*}
\emph{or}
\begin{equation*}
\Card{X[N]}\leq (3-\tfrac{1}{g})N^2+(8-\tfrac{1}{g})N+4g.
\end{equation*}
\end{theorem}

By taking more care, we can obtain stronger results in some cases.  For example, when $X$ is hyperelliptic and $\infty$ is a Weierstrass point, we recover the bound of Theorem~1.2 of \cite{Bo23} (see subsection \ref{subsecHCWP}).  When $g=2$ and $N$ is large, the above bound is weaker (when it applies) than the bound of Proposition~\ref{prop:oldbound}. When $g=2$ and $k=\bC$, Pareschi \cite{Par21} has proved that $\Card{X[N]}\leq \frac{3}{2}N^2$ for all $N$ whatever the base point $\infty$; in particular our question has a positive answer in this case.  Although our methods lead to a slightly weaker result, they also apply in positive characteristic (see subsection \ref{subsecg2notWP} for details). 

Under certain hypotheses, the first possibility cannot occur: 

\begin{prop} \label{prop:wellknown} 
Suppose that $N\geq 2g-1$. If either $p=0$ or $N<p$, then
\begin{equation*}
(N-g+1)_*X\not\subseteq  W_{g-1}^-.
\end{equation*}
\end{prop}

\begin{cor}
In characteristic zero, we have $M_*X\not\subseteq W_{g-1}^-$ for all $M\geq g$. 
\end{cor}

This is well-known; it can be viewed, for example, as a special case of results of Laksov~\cite{La81}. When $X$ is hyperelliptic and $\infty$ a Weierstrass point, it was proved by Cantor~\cite{Ca94}. We prove a mild generalization, relevant to the context of torsion points, in \S~\ref{sec:DeltaNzero}.

The hypothesis $M\geq g$ is necessary the Corollary. For example,  when $X$ is hyperelliptic and $\infty$ a Weierstrass point, then $W_{g-1}^-=W_{g-1}$ and obviously $M_*X\subseteq W_{g-1}$ for all $M\leq g-1$. As another example, suppose that $X$ has an automorphism of order $3$ fixing $\infty$. This automorphism induces an endomorphism $\lambda$ of $J$ that satisfies $\lambda^2+\lambda+1=0$. Since also $\lambda_*X=X$, we have $X\subseteq W_2^{-}$, from which it follows that $M_*X\subseteq W_{g-1}^-$ for all $M\leq \frac{g-1}{2}$. 

In characteristic $p>0$ the Proposition (and the Corollary) can also fail. One easy way to construct a counterexample is to take $X$ defined over a finite field with the property that some power of the Frobenius endomorphism acts as $-p^t$ for some integer $t>0$. Then $-p^t_*X=X$ and so trivially $p^{tn}_*X\subseteq W_{g-1}^-$ for all $n\geq 1$.  

In the terminology introduced by Neeman \cite{Ne84}, the condition $(N-g+1)_*X\subseteq W_{g-1}^-$ is equivalent to (the line bundle associated to) the divisor $N[\infty]$ being \emph{weird}. In the above example, as well as those discussed by Neeman which are similar, the jacobian $J$ of $X$ has the property of being isogenous to a power of a supersingular elliptic curve (see also Example~\ref{ex:MNrnotmax}). A example of a curve with $N[\infty]$ weird but $J$ not isogenous to a power of a supersingular elliptic curve will be given in Example~\ref{ex:g9family}.  Note that the dichotomy in Theorem~\ref{theo:intromain} is already alluded to in \cite{Ne84} (see for example Fact~5.1). 

The method of proof of Theorem~\ref{theo:intromain} is a generalization of that of the main result of \cite{Bo23}.  However, instead of working with polynomials we work with elements of the function field $k(X)$ of $X$. Let $d$ be the smallest positive integer such that $k(X)$ contains a function $x$ with a pole of order $d$ at $\infty$ and is finite elsewhere, let $\pi:X\to \bP^1$ be the morphism defined by $x$. One checks that $\pi$ is separable, so let $\cR$ be the set of points of $X$ where $\pi$ is ramified. Set $\rho=\Card{\cR}$.  For each $N$, we construct a function $\Delta_N\in k(X)$ which vanishes if and only if $(N-g+1)_*X\subseteq W^-_{g-1}$ (see Definition~\ref{def:minors}). In fact $\Delta_N$ is defined to be the determinant of a wronskian matrix, which is itself a submatrix of matrix with $N-g+1$ rows and $N$ columns which we denote by $\cM_N$. The crucial point is Proposition~\ref{prop:crucial} and its corollaries. However in order to work in positive characteristic, we use Hasse-Schmidt derivatives rather than iterates of usual derivatives. This is the classical strategy in the study of Weierstrass points, that has already been adapted to the positive characteristic case in, for example,  the work of St\"{o}hr and Voloch \cite{StoVol86} and Laksov \cite{La81}. However its application to the study of torsion points seems to be new. When $\Delta_N\neq 0$, its poles all belong to $\cR$ and we compute a bound on the sum of the orders of its poles. Since the degree of the divisor of a function is zero, this leads to a bound on the sum of the orders of the zeros of $\Delta_N$.

On the other hand, we shall show that $\Delta_N$ has a zero of order at least $g$ at every point of $X[N]$ that does not belong to $\cR$. When $\Delta_N\neq 0$, this leads to a bound on $\Card{(X[N]-\cR)}$ in terms of $d$,  $g$ and the ramification invariants associated to the morphism $\pi$.  See Theorem~\ref{theo:main} for the precise statement. Simplifying leads to the bound in Theorem~\ref{theo:intromain}.

As in Proposition~1.5 of \cite{Bo23}, the fact that  $\Delta_N$ has a zero of order at least $g$ at every point of $X[N]$ that does not belong to $\cR$ has the following generalization.

\begin{prop} \label{prop:Nplusrzeros} Let $r$ be a non-negative integer such that $g-r+\lfloor{\frac{r}{d}}\rfloor>0$ and let $\xi\in X[N]-\cR$. Suppose $N\geq 2g-1$. Then $\Delta_{N+r}$ has a zero of order at least
\begin{equation*}
\Big(g-r+\Big\lfloor{\frac{r}{d}}\Big\rfloor\Big)\Big(\Big\lfloor{\frac{r}{d}}\Big\rfloor  +1 \Big)
\end{equation*} 
at $\xi$. 
\end{prop}

We refer to Proposition~\ref{prop:orderzeros} for a more general result. In the final \S~\ref{sec:applex}, we discuss small values of $N$, show how to deduce Theorem~\ref{theo:intromain} from the results in \S\S~\ref{sec:vals}--\ref{sec:DeltaNzero}, then show how to recover most of the results of \cite{Bo23} and finally discuss a few examples. 

In characteristic $0$, the Manin-Mumford conjecture says that $X_{\tors}:=\bigcup_{N\geq 2}X[N]$ is a finite set, so $\Card{(X[N])}$ is bounded  independently of $N$. Since the first proof of the conjecture by Raynaud in 1983 \cite{Ra83}, many others have appeared, and recent work (see for example the papers of Looper, Silverman and Wilms \cite{LoSiWi21},  K\"uhne \cite{Ku21} and Yuan \cite{Yu23}) has led to uniform bounds on $\Card{X_{\tors}}$ depending only on $g$ when $k$ is a function field or has characteristic $0$.  In particular, Looper, Silverman and Wilms give the remarkable quadratic bound $\Card{X_{\tors}}\leq 16g^2+32g+124$ (and in fact a stronger bound) for all non-isotrivial $X$ defined over the function field $k_0(B)$ of a smooth projective curve $B$ defined an algebraically closed field $k_0$ of arbitrary characteristic. This obviously gives a positive answer to our question in this case. On the other hand, the papers of K\"uhne and Yuan (which study curves over $\overline{\bQ}$) does not seem to lead to such an explicit bound. Finally, our bounds (for example Proposition~\ref{prop:oldbound}) are linear in $g$ so are better than that of Looper, Silverman and Wilms when $N$ is small with respect to $g$. 

When $X$ is defined over a finite field, all points of $X$ defined over an algebraic closure have finite order, $X_{\tors}$ is infinite and $\Card{(X[N])}$ is unbounded. Therefore bounds on $\Card{(X[N])}$ are of particular significance in this case. This justifies that in this paper we concentrate on methods that work over an algebraically closed field of arbitrary characteristic. 

As in \cite{Bo23}, our results lead to an efficient method for computing $\Delta_N$ and $X[N]$ in simple cases, and the examples at the end of \S~\ref{sec:applex} illustrate this.  

If $X$ is defined over a number field $K$ and $\infty\in X(K)$, then the absolute Galois group $G_K$ of $K$ acts on $X[N]$ and on $J[N]$; thus Theorem~\ref{theo:intromain} implies a bound on the degree of the extension $K(X[N])/K$ and the presence of a point of order $N$ on $X$ puts restrictions on the image of the natural representation of $G_K$ in the automorphism group of $J[N]$. Although we do not persue this here, it was this problem that led to the author's interest in the subject.   

\begin{rem} When $X$ is hyperelliptic and the base point $\infty$ a Weierstrass point, the function $\Delta_N$
 of the paper is not the same as the polynomial $\Delta_N$ defined in \cite{Bo23}. In fact, if $y^2+Q(x)y=P(x)$ is an equation for $X$, then the $\Delta_N$ of \cite{Bo23} is the product of our $\Delta_N$ and a power of $2y+Q(x)$, as follows from easily from Lemma~2.1 and the definitions of \cite{Bo23}. Since $2y+Q(x)$ vanishes precisely at the finite ramification points of the morphism $X\to \bP^1$ defined by $x$, the two definitions can be viewed as equivalent for the purposes of this paper.  For example, their vanishing is independent of the definition used. Similar remarks apply to other notation; for example the matrix $\cM_N(X)$ of this paper is not the same as the matrix $M_N(X)$.
\end{rem}

\section{Lower bounds on valuations at ramified points} \label{sec:vals}

We return to the notation of the Introduction. For the moment, we allow $d$ to be \emph{any} positive integer such that there exists a function $x\in k(X)$ finite away from $\infty$ and with a pole of order $d$ at $\infty$, and such that the field extension $k(X)/k(x)$ is separable. Since $g\geq 1$, we have $d\geq 2$. Let $\pi:X\to \bP^1$ be the morphism defined by $x$ and $\cR$ the set of ramification points of $\pi$.  Of course $d$ is also the degree of $\pi$. Note that $\infty\in \cR$. Set $\rho=\Card{\cR}$ and denote by $e$ (or $e_\xi$  if we need to refer explicitly to $\xi$) the ramification degree of $\pi$ at $\xi$. If $\xi\in \cR$, denote by $t=t_\xi$ a uniformizer at $\xi$. When $\xi\neq \infty$, $x$ can be viewed as an element of the completed local ring $\hat{ \cO}_{X,\xi}$ and has an expansion
\begin{equation} \label{eq:xdeltxi}
x=x(\xi)+\sum_{m\geq e}b_mt^m,  \quad   b_m\in k,{\ } b_{e}\neq 0,
\end{equation}
For simplicity of notation, we assume that $b_e=1$. Similarly, $x^{-1}$ is finite at $\infty$, $e_\infty=d$ and there is an expansion
\begin{equation} \label{eq:xdeltinfty}
\frac{1}{x}=\sum_{m\geq d}b_mt^m,  \quad b_m\in k,{\ } b_d\neq 0,
\end{equation}
where again we assume that $b_d=1$.

Let $r=r_\xi$ be the smallest integer prime to $p$ such that $b_r\neq 0$. This exists, since $\pi$ is supposed separable, and $r$ is easily seen to be independent of the choice of the uniformizer $t$.  We have $r\geq e$, and $\pi$ is tamely or wildly ramified at $\xi$ according as to whether $r=e$ or $r>e$. Of course $\pi$ is always tamely ramified when $p=0$.

\begin{lemma} \label{lem:RH} We have
\case{i} 
\begin{equation} \label{eq:Hurwitz}
2(g+d-1)=\sum_{\xi\in \cR} (r_\xi-1)
\end{equation}
and
\case{ii} $\rho\leq 2(g+d-1)$.
\end{lemma}

\begin{proof} \case{i} Using (\ref{eq:xdeltxi}), we find, if $\xi\in \cR$ and $\xi\neq \infty$ and viewing $x$ as a function on $\bP^1$:
\begin{equation*}
\pi^*\diff{x}=\sum_{m\geq r}mb_mt^{m-1}\, \diff{t}= \big(rb_rt^{r-1}+O(t^r)\big)\,\diff{t}
\end{equation*}
and $r b_{r}\neq 0$. Similarly, when $\xi=\infty$ and using (\ref{eq:xdeltinfty}):
\begin{equation*}
\pi^*\diff{x}=-\pi^*x^2\diff{\Big(\frac{1}{x}\Big)}=\big(-rb_rt^{r-1-2d}+O(t^{r-2d})\big) \diff{t}=-t^{-2d}\big(rb_rt^{r-1}+O(t^{r})\big) \diff{t}, 
\end{equation*}
and again $rb_r\neq 0$.

But $\diff{x}$ is a canonical divisor on $\bP^1$ with divisor $-2[\pi(\infty)]$, $\pi$ has degree $d$ and  $\pi^*(-2[\pi(\infty)])=-2d[\infty]$. Since $\bP^1$ has genus $0$, it follows that (\ref{eq:Hurwitz}) is just the Hurwitz formula applied to $\pi$ (see for example \cite{Liu02}, \S~7.4.2).

\case{ii} Since $r_\xi\geq e_\xi\geq 2$ for all $\xi\in \cR$, $\rho\leq \sum_{\xi\in \cR}(r_\xi-1)$ and the bound follows from \case{i}.
\end{proof}

Now let $D=(D_n)_{n\geq 0}$ be the standard Hasse-Schmidt derivative on $k(x)$ trivial on $k$, so each $D_n$ is $k$-linear and $D_n(x^m)=\binom{m}{n}x^{m-n}$ for all $m\in \bZ$ and $n\geq 0$. Furthermore, $D$ has the crucial translation-invariance property
\begin{equation*}
D_n(x-x_0)^m=\binom{m}{n}(x-x_0)^{m-n}
\end{equation*}
for all $x_0\in k$, $m\in \bZ$, $n\geq 0$.  This implies that $D$ has a unique extension to the fraction field $K_{\bP^1,x_0}$ of each completed local ring $\hat{\cO}_{\bP^1,x_0}$, continuous for the $M$-adic topology associated to the maximal ideal $M$, and leaving invariant $\hat{\cO}_{\bP^1,x_0}$ invariant. 

If $D=(D_n)_{n\geq 0}$ is any Hasse-Schmidt derivation, we have the formula
\begin{equation*}
D_n(fg)=\sum_{\ell=0}^nD_\ell(f)D_{n-\ell}(g)
\end{equation*}
for all $n\geq 0$. We deduce that for all $n\geq 0$ and $m\geq 1$ we have
\begin{equation} \label{eq:Dnfm}
D_n(f^m)=\sum_{\ell_1+\ell_2+\cdots +\ell_m=n}\prod_{i=1}^mD_{\ell_i}(f)
\end{equation}
where the sum runs over all $m$-tuples of non-negative integers $(\ell_1,\ell_2,\dots, \ell_m)$ whose sum is $n$.  When $n=1$, this reduces to $D_1(f^m)=mf^{m-1}D_1(f)$ which also holds when $m<0$ as one sees by applying $D_1$ to the formula $f^{-m}f^m=1$. 

\begin{lemma} \label{lem:extD}
Let $U=X-\cR$ and let $\xi\in X$.

\case{i} $D$ has unique extensions to $k(X)$ and to the fraction field $K_{X,\xi}$ of the completed local ring $\hat{\cO}_{X,\xi}$ and these extensions are compatible with the field inclusions  $k(x)\subseteq k(X)\subseteq K_{X,\xi}$ and $k(x)\subseteq K_{\bP^1,\pi(\xi)}\subseteq K_{X,\xi}$ induced by $\pi$.

\case{ii} If $\xi\in U$, then $D$ leaves $\cO_X(U)$ and $\hat{\cO}_{X,\xi}$ invariant.
\end{lemma}

\begin{proof}
\case{i} Since $D_1x=1$, the extension of the usual derivation $D_1$ to $k(X)$ satisfies $D_1^p=0$, so $D$ extends to a Hasse-Schmidt derivative on $k(X)$ (see \cite{Mats86}, \S~27).  Similarly, $D$ extends to $K_{X,\xi}$ since this is necessarily a separable extension of $K_{\bP^1,\pi(\xi)}$. To see that it is unique, let $f\in k(X)$ and let $F(x,T)=\sum_{m=0}^Ma_m(x)T^m$ be the minimal polynomial of $f$ over $k(x)$, where $a_M(x)=1$.  Note that $\frac{\partial F}{\partial T}\neq 0$ since $k(X)/k(x)$ is separable. Thus $\sum_{m=0}^Ma_m(x)f^m=0$ and since $D_1$ is a derivation a standard computation shows there is only one possible formula for $D_1(f)$.  To deal with higher values of $n$, note that (\ref{eq:Dnfm}) can be rewritten as
\begin{equation} \label{eq:primedsum}
D_n(f^m)=mf^{m-1}D_n(f)+ {\sum}'_{(\ell)}\prod_{i=1}^mD_{\ell_i}(f)
\end{equation}
where the primed sum is now over those $(\ell)= (\ell_1,\ell_2,\dots, \ell_m)$ for which $\ell_1+\ell_2+\cdots +\ell_m=n$ and $\ell_i\neq 0$ for at least two indices $i$.  Then $\ell_i\leq n-1$ for all $i$.  Applying this to $D_n\big(\sum_{m=0}^Ma_m(x)f^m\big)=0$ shows by induction on $n$ that there is only one possible formula for $D_n(f)$. 

The same proof shows the uniqueness of the extension of $D$ from $K_{\bP^1,\pi(x)}$ to $K_{X,\xi}$, and this implies compatibility with the field inclusions. 

\case{ii} Since $X$ is smooth, $\cO_X(U)$ is integrally closed, $\cO_X(U)=\bigcap_{\xi\in U} \cO_{X,x}$ and $\cO_{X,x}=k(X)\cap \hat{\cO}_{X,\xi}$ for all $\xi\in U$. Therefore it suffices to show that $D$ leaves $\hat{\cO}_{X,\xi}$ invariant.  Since $\xi\in U$, $\pi$ is unramified at $\xi$ and, since $k$ is algebraically closed, the inclusion $\hat{\cO}_{\bP^1,\pi(\xi)}\subseteq \hat{\cO}_{X,\xi}$ is an equality.  But we have already noted that $\hat{\cO}_{\bP^1,\pi(\xi)}$ is $D$-invariant. 
\end{proof}

Since $X$ is smooth, $\cO_{X,\xi}$ is a discrete valuation ring for all $\xi\in X$ and we write $v_\xi$ (or just $v$ if there is no risk of confusion) for the associated valuation. Thus $v_\xi(t_\xi)=1$.

\begin{prop} \label{prop:valDnf} Let $f\in \cO_X(U)$. Then $D_n(f)\in \cO_X(U)$ for all $n\geq 1$ and, if $\xi\in \cR$, then
\begin{equation*}
v_\xi(D_n(f))\geq \begin{cases} (n-1)e_\xi-(2n-1)r_\xi+v_\xi(f) &\text{ when $\xi\neq \infty$}\\
(3n-1)d-(2n-1)r_\infty+v_\infty(f) &\text{ when $\xi=\infty$} \end{cases}
\end{equation*}
for all $n\geq 1$.
\end{prop}

\begin{proof} The first assertion follows from Lemma~\ref{lem:extD}, so we concentrate on the proof of the lower bounds. The case $f=0$ is trivial, so we suppose throughout that $f\neq 0$. Let $v(f)=m$, so $v(f)=v(t^m)$. Since the bounds are strictly increasing functions of $v(f)$, it suffices to prove them when $f=t^m$.   When $m=0$, they are trivial since $D_n(1)=0$ for all $n\geq 1$. We consider successively the cases $m=1$, $m\geq 2$ and $m<0$, dealing first with the case $\xi\neq \infty$. 

\case{a} The case $m=1$, $\xi\neq \infty$.  Applying $D_1$ to $x=\sum_{m\geq e}b_mt^m$ gives $1=D_1(t)(rb_rt^{r-1}+O(t^r))$. Since $v(1)=0$, we deduce that $v(D_1(t))=-v(rb_rt^{r-1})=1-r$ since $rb_r\neq 0$. Hence equality holds when $n=1$. 

Suppose now that $n\geq 2$ and that the result holds when $n$ is replaced by an integer $\ell$ such that $1\leq \ell<n$. Now
 \begin{align*}
 0=D_n(x-x(\xi))&=\sum_{m\geq e}b_mD_n(t^m)\\
                       &=\sum_{m\geq e}b_m\Big(\sum_{\ell_1+\ell_2+\cdots +\ell_m=n}\prod_{i=1}^mD_{\ell_i}(t)\Big)\\
                       &=\sum_{m\geq e}mb_mD_n(t)t^{m-1}+\sum_{m\geq e}b_m\Big({\sum_{(\ell)}}'\prod_{i=1}^mD_{\ell_i}(t)\Big)\\
                       &=D_n(t)(rb_rt^{r-1}+O(t^r))+\sum_{m\geq e}b_m\Big({\sum_{(\ell)}}'\prod_{i=1}^mD_{\ell_i}(t)\Big).
  \end{align*}
where the primed sum is as in (\ref{eq:primedsum}).  We now estimate $v\Big(\prod_{i=1}^mD_{\ell_i}(t)\Big)$ for each of the terms inside the primed sum. Fix one such term and let $s$ be the number of indices $i$ such that $\ell_i=0$. Then since $v(t)=1$, $v(D_{\ell_i(t)})\geq 1+(\ell_i-1)e-(2\ell_i-1)r$ and $\sum_{i}\ell_i=n$, we have 
\begin{align*}
v\Big(\prod_{i=1}^mD_{\ell_i}(t)\Big)=\sum_{i=1}^mv(D_{\ell_i}(t))&\geq s+\sum_{i\mid \ell_i\geq 1}(1+(\ell_i-1)e-(2\ell_i-1)r)\\
                                                                                                                  &= s+(m-s)+(n-m+s)e-(2n-m+s)r\\
                                                                                                                  &= m+ne-2nr+(m-s)(r-e).
\end{align*}
This is true for all $m\geq e$ and the smallest value that occurs is when $m=e$. Note that since no $\ell_i$ is equal to $n$, $s\geq m-2$ so $m-s\geq 2$. Also, $r\geq e$. Hence for all $m\geq e$:
\begin{equation*}
v\Big(\prod_{i=1}^mD_{\ell_i}(t)\Big)\geq e+ne-2nr+2(r-e)=(n-1)e-(2n-2)r.
\end{equation*}  
It follows that $v(D_n(t)rb_rt^{r-1})\geq (n-1)e-(2n-2)r$ and therefore
\begin{equation*}
v(D_n(t))\geq (n-1)e-(2n-2)r+1-r=1+(n-1)e-(2n-1)r
\end{equation*}
as required.  

\case{b} The case $m>1$, $\xi\neq \infty$. Then, using the same notation and argument as in \case{a} we find
\begin{equation*}
D_n(t^m)=\sum_{(\ell)}\prod_{i=1}^mD_{\ell_i}(t)
\end{equation*}
and so, for any $(\ell)=(\ell_1,\ell_2,\dots, \ell_m)$,
\begin{equation*}\sum_{i=1}^m{v(D_{\ell_i}(t))}\geq m+ne-2nr+(m-s)(r-e).
\end{equation*}
The minimal estimate occurs when $m-s$ is as small as possible, which happens when $s=m-1$. Then $m-s=1$ and the result follows.

\case{c} The case $m<0$, $\xi\neq \infty$. Put $q=-m$, so $q>0$ and $m=-q$. From $t^{-q}t^{q}=1$ and $n\geq 1$ we deduce $D_n(t^{-q}t^q)=0$, so that $\sum_{\ell=0}^nD_{\ell}(t^{-q})D_{q-\ell}(t^q)=0$, whence $D_n(t^{-q})=-t^{-q}\sum_{\ell=0}^{n-1}D_{\ell}(t^{-q})D_{n-\ell}(t^q)$.  When $n=1$, this reduces to $D_1(t^{-q})=-t^{-2q}D_1(t^q)$, so since $v(t)=1$ and $v(D_1(t^q))\geq q-r$, we find $v(D_1(t^{-q}))\geq -2q+q-r=-q-r$ as required. Equality holds when $q=n=1$ since  $v(D_1(t))=1-r$. Suppose that $n\geq 2$ and that the result is true when $n$ is replaced by any integer $\ell<n$. Then, if further $\ell\geq 1$,
\begin{align*}
v(D_{\ell}(t^{-q})D_{n-\ell}(q))&\geq(-q+(\ell-1)e-(2\ell-1)r)+(q+(n-\ell-1)e-(2n-2\ell-1)r\\
                                                   &=(n-2)e-(2n-2)r
\end{align*}
while
\begin{equation*}
v(t^{-q}D_n(t^q))\geq -q+q+(n-1)e-(2n-1)r=(n-1)e-(2n-1)r.
\end{equation*}
Since $r\geq e$, the term $v(t^{-q}D_n(t^q))$ contributes the smallest lower bound and we deduce that $v(t^qD_n(t^{-q}))\geq (n-1)e-(2n-1)r$ as claimed.

\case{d} The case $\xi=\infty$ (and all $m$). The argument is basically the same as when $\xi\neq \infty$. Recall that $v_\infty(\frac{1}{x})=d$. From 
\begin{equation*}
-\frac{1}{x^2}=D_1\Big(\frac{1}{x}\Big)=\sum_{m\geq d}D_1(b_mt^{m})=D_1(t)\sum_{m\geq d}b_mmt^{m-1}=D_1(t)(b_rrt^{r-1}+O(t^r)),
\end{equation*}
we find that $v_\infty(D_1(t))=2d-r+1$.  To do the case $m=1$ and $n\geq 2$, we start with
\begin{align*}
(-1)^n\frac{1}{x^{n+1}}&=D_n\Big(\frac{1}{x}\Big)=\sum_{m\geq d}b_mD_n(t^m)\\
                                   &=D_n(t)(rb_rt^{r-1}+O(t^r))+\sum_{m\geq d}b_m\Big({\sum_{(\ell)}}'\prod_{i=1}^mD_{\ell_i}(t)\Big).
\end{align*}
Thus a lower bound for $v_\infty(t^{r-1}D_n(t))$ is given by the minimum of  $v_\infty(x^{-n-1})=(n+1)d$ and $\sum_{i=1}^mv_\infty(D_{\ell_i}(t))$ for $(\ell)=(\ell_1,\ell_2,\dots  ,\ell_m)$ as in the primed sum.  Fix such an $(\ell)$ and let $s$ be the number of indices $i$ such that $\ell_i=0$. Then
\begin{align*}
\sum_{i=1}^mv_\infty(D_{\ell_i}(t))&\geq s+\sum_{i\mid \ell_i\geq 1}(3\ell_i-1)d-(2\ell_i-1)r+1)\\
                                                     &=s+(3n-m+s)d-(2n-m+s)r+(m-s)\\
                                                     &=m+(m-s)(r-d)+3nd-2nr. 
\end{align*}
Hence, since $0\leq s\leq m-2$, $m\geq d$ and $r\geq d$:
\begin{align*}
v_\infty(D_n(t))&\geq \min{((n+1)d-r+1, d+2(r-d)+3nd-2nr-r+1)}\\
                        &=\min{((n+1)d-r+1, (3n-1)d-(2n-1)r+1)}.
\end{align*}
Since $r\geq d$, the minimum is attained at $(3n-1)d-(2n-1)r+1$. This completes the proof when $m=1$ and $n\geq 1$.

The cases $m\geq 2$ and $m<0$ then follow as before. \end{proof}

\section{The matrix $\cM_N$} \label{sec:matrixMn}

We continue to use the notation introduced in the previous section.  Let $\delta_1<\delta_2\cdots  <\delta_n<\cdots$ be the strictly increasing sequence orders of poles at $\infty$ of functions in $\cO_X(X-\{\infty\})$. Thus $n=\dim_k(\cO_X(m[\infty]))$ whenever $\delta_n\leq m< \delta_{n+1}$. Since $1\in \cO_X(X-\{\infty\})$, we have $\delta_1=0$. By the Riemann-Roch theorem, $\delta_n=n+g-1$ for all $n\geq g+1$, so $\delta_{n+1}=\delta_n+1$ for these $n$. The set $\{\delta_n\}_{n\geq 1}$ form a sub-semigroup of the additive semigroup $\bZ_+$. The positive integers not of the form $\delta_n$ for some $n$ are known as the Weierstrass gaps at $\infty$, it is well-known that there are always $g$ of them.

Fix $N\geq 2$ as in the Introduction. We define $\sD_N=\dim_{k}(\cO_X(N[\infty]))$, so $\sD_N$ is the unique integer $n$ such that $\delta_n\leq N<\delta_{n+1}$. By the Riemann-Roch theorem, $0\leq N-\sD_N\leq g-1$ for all $N$, and $\sD_N=N-g+1$ when $N\geq 2g-1$. 

For each $n$, we choose a function $y_n\in \cO_X(X-\{\infty\})$ such that $v_\infty(y_n)=-\delta_n$.  Once such a choice is made, we define a matrix $\cM_N=\cM_N(X)$ with $\sD_N$ rows and $N$ columns and entries in $k(X)$ by
\begin{equation*}
\big(\cM_N(X)\big)_{i,j}=D_{j-1}(y_i), \qquad 1\leq i\leq \sD_N, \quad 1\leq j\leq N.
\end{equation*}

In general, if $\cM$ is a matrix and $r\in \bZ_{\geq 0}$, we denote by $\cM^{(r)}$ the matrix obtained by suppressing the rightmost $r$ columns of $\cM$. In particular, $\cM^{(0)}=\cM$. 

Recall that $U=X-\cR$. By Proposition~\ref{prop:valDnf}, the entries of $\cM_N$ actually lie in $\cO_X(U)$. If $\xi\in U$, we denote by $\cM_N(\xi)$ the matrix obtained by evaluating each entry of $\cM_N$ at $\xi$. Similarly for $\cM_N^{(r)}(\xi)$. In what follows, we suppose that $0\leq r\leq N-\sD_N$. Since $\sD_N\leq N-r$, the rank of $\cM_N^{(r)}$ is at most $\sD_N$, and this is also the maximal possible value of the rank $\cM_N^{(r)}(\xi)$ when  $\xi\in U$. Note also that $\cM_N^{(N-\sD_N)}$ is a square matrix. The following Proposition is central to all that follows.

\begin{prop}  \label{prop:crucial} Let $\xi\in U$ and let $r\in \{0,1,\dots, N-\sD_N\}$. Then the following are equivalent:

\case{a} $(N-r)(\xi)\in  W^-_r$,

\case{b} $\cM_N^{(r)}(\xi)$ has rank strictly less than $\sD_N$.
\end{prop}

\begin{proof} By the Abel-Jacobi theorem, $(N-r)\xi\in W^-_r$ if and only the divisor $(N-r)([\xi]-[\infty])$ is linearly equivalent to a divisor of the form $-(\sum_{i=1}^r[\eta_i]-r[\infty])$, in other words if and only if there exists $\alpha\in k(X)$ with divisor $(N-r)[\xi]+\sum_i[\eta_i]-N[\infty]$.  Such an $\alpha$ lies in $\cO_X(N[\infty])$, so can be written $\alpha=\sum_{i=1}^{\sD_N}\lambda_iy_i$ for some $\lambda_i\in k$ not all $0$. On the other hand, since $\xi\notin \cR$, $\alpha$ having a zero of order at least $N-r$ at $\xi$ is equivalent to $D_j(\alpha)$ vanishing at $\xi$ for all $j\in \{0,1,\dots, N-r-1\}$. Thus if $(N-r)\xi\in W^-_r$, there is a linear relation $\sum_{i}\lambda_iD_j(y_i)(\xi)=0$ among the rows of $\cM_N^{(r)}(\xi)$.  Conversely, if such a linear relation exists, and we define $\alpha=\sum_{i}\lambda_i y_i$, then $\alpha\in \cO_X(N[\infty])$ and it vanishes to order at least $N-r$ at $\xi$ ($\alpha$ is non-zero since the $y_i$'s are linearly independent). Since the degree of the divisor of a function is zero, the divisor of $\alpha$ must be of the form $(N-r)[\xi]+\sum_i[\eta_i]-N[\infty]$ and hence $(N-r)\xi\in W^-_r$. 
\end{proof}

\begin{cor} \label{cor:crucial}
Let $\xi\in U$. Then $\xi \in X[N]$ if and only if the rank of $\cM_N(\xi)$ is strictly less than $\sD_N$.
\end{cor}

\begin{proof}
This is just the case $r=0$ of the Proposition.
\end{proof}

\begin{cor} \label{cor:crucialgeneric}
The matrix $\cM_N$ has rank $\sD_N$.  Suppose that $r\in \{1,2,\dots, N-\sD_N\}$.  Then $\cM_N^{(r)}$ has rank $\sD_N$ if and only if $(N-r)_*X\not\subseteq W_r^-$.
\end{cor}

\begin{proof}
Since $\cM_N$ has $\sD_N$ rows, its rank is at most $\sD_N$. If its rank were strictly smaller, then since $k$ is algebraically closed and $\cR$ is a finite set, it would follow from Corollary~\ref{cor:crucial} that $X[N]$ were an infinite set. But $X[N]\subseteq J[N]$, which is a finite set. Since $\cR$ is a finite set, the second statement follows from Proposition~\ref{prop:crucial}.  
\end{proof}

\section{Bounds on zero divisor degrees}

It follows from Proposition~\ref{cor:crucial} that points of $X[N]$ not belonging to $\cR$ correspond to common zeros of the $\sD_N\times \sD_N$ minors of $\cM_N$. We index these by their columns: 

\begin{notation} \label{def:minors}For each sequence $s=(s_1,s_2,\dots , s_{\sD_N})$ of $\sD_N$ integers such that $1\leq s_1<s_2<\cdots < s_{\sD_N}\leq N$, we denote by $\Pi_{N,s}$ the minor obtained from columns $s_1$, $s_2$, \dots, $s_{\sD_N}$ of $\cM_N$. We define $\Delta_N$ to be the leftmost minor ($s_i=i$ for all $i$). 
\end{notation}

In what follows, we shall allow ourselves to permute the lines of $\cM_N$, so the $\Pi_{N,s}$'s should be viewed as only defined up to sign. By the Corollary, we know that at least one of the $\Pi_{N,s}$'s is not identically zero. In this section we shall prove bounds on the number of zeros in $U$ of the $\Pi_{N,s}$'s under the assumption that they are not identically zero. 

The matrix $\cM_N$ depends on the choice of the functions $y_n$; we shall now explain how to choose them in a way to make the computations simpler.

We fix once and for all the uniformizer $t=t_\infty$ at $\infty$ and we always assume that 
\begin{equation*}
y_n=\frac{1}{t^{\delta_n}}+O\big(\frac{1}{t^{\delta_n-1}}\big)
\end{equation*}
for all $n$. This is obviously just a matter of rescaling. 

For each residue class $\gamma \pmod{d}$, we fix a choice $z_\gamma$ of $y_n$, where $n$ is the smallest integer such that $\delta_n\in \gamma$, and then take $y_{n+id}=x^iz_\gamma$ for all $i=0$, $1$, $2$, \dots {}.   We always assume that $z_0=1$. We set $N_\gamma=\Card{\{n\leq N \mid \delta_n\in \gamma\}}$. Thus $\sum_{\gamma\in \bZ/d\bZ}N_\gamma=\sD_N$.  If we extract the $n^{\text{th}}$ rows of $\cM_N$ with $n$ running over the integers up to $\sD_N$ such that $\delta_n\in \gamma$, we obtain a matrix with $N_\gamma$ rows
\begin{equation*}
\cN_{N,\gamma}=\begin{pmatrix} z & D_1(z) & D_2(z) & \cdots  &  D_{N-1}(z) \\
xz       &  D_1(xz) & D_2(xz)   &    \cdots        &  D_{N-1}(xz)          \\
x^2z   &D_1(x^2z)& D_2(x^2z) &  \cdots        & D_{N-1}(x^2z)       \\
\vdots &\vdots      &  \vdots      &   \ddots        &   \vdots                  \\
x^mz & D_1(x^mz)& D_2(x^mz)&  \cdots       &  D_{N-1}(x^mz)   
\end{pmatrix},
\end{equation*}
where $z=z_\gamma$ and $m=N_\gamma-1$. 

We define an $\sD_N\times N$ matrix $\cN_N$ by stacking the matrices $\cN_{N,0}$, $\cN_{N,1}$, \dots , $\cN_{N,d-1}$ each just above the next. Thus $\cN_N$ is obtained from $\cM_N$ by a permutation of the rows.

Next, let $\cN'_{N,\gamma}$ be the $N_\gamma\times N$ matrix defined by
\begin{equation*}
\cN'_{N,\gamma}=\begin{pmatrix} z & D_1(z) & D_2(z) & \cdots  & D_{N-1}(z) \\
                  0  &  z  &  D_1(z)  &  \cdots    &   D_{N-2}(z)   \\
                  0  &  0  &    z        &  \cdots     &   D_{N-3}(z)  \\
         \vdots & \vdots & \vdots &   \ddots   &    \vdots        \\
                 0   &   0  &   0        &   \cdots     &  D_{N-N_{\gamma}}(z)      
\end{pmatrix}. 
\end{equation*}
where again $z=z_\gamma$. (More formally, $\big(\cN'_{N,\gamma}\big)_{i,j}= D_{j-i}(z)$,  where $ 1\leq i\leq N_\gamma$, $1\leq j\leq N$,  and $D_{j-i}(z)=0$ if $j<i$.) Finally, we define $\cN'_N$ to be the $\sD_N\times N$ matrix obtained by stacking the matrices $\cN'_{N,\gamma}$ each above the next in the same way that $\cN_N$ was defined using the $\cN_{N,\gamma}$s.  For each sequence $s=(s_1,s_2,\dots , s_{\sD_N})$ with $1\leq s_1<s_2<\cdots < \sD_N\leq N$, denote by $\Pi'_{N,s}$ the minor of $\cN'_N$ defined by columns $s_1$,  $s_2$, \dots, $s_{\sD_N}$. 

\begin{lemma} \label{lem:PiPiprime}Up to sign, we have $\Pi_{N,s}=\Pi'_{N,s}$.
\end{lemma}

\begin{proof}
We have already remarked that $\cN_N$ is obtained from $\cM_N$ by a permutation of the rows, so it remains to prove that  $\cN_N$ and $\cN'_N$ have the same minors.  But if $i\geq 0$, then
\begin{equation*}
D_n(x^iz)=\sum_{j=0}^{i-1}\binom{i}{j}x^{i-j}D_{n-j}(z)  + D_{n-i}(z)
\end{equation*}
from which we deduce that, if $\gamma\in \bZ/d\bZ$ and if $\cL_i$ and $\cL'_i$ denote respectively the $i^{\text{th}}$ lines of $\cN_{N,\gamma}$ and $\cN'_{N,\gamma}$, then $\cL_i=\cL'_i+\sum_{j=1}^{i-1}\binom{i-1}{j}x^j\cL'_{i-1-j}$.  The result follows.
\end{proof}

We can now prove the general bound on the number of zeros of $\Pi_{N,s}$ in $U$. For $\gamma\in \{-1,0,1,\dots  ,d-1\}$ define $S_\gamma$ by $S_{-1}=0$ and $S_\gamma=S_{\gamma-1}+N_\gamma$ if $\gamma\geq 0$. Thus $S_{d-1}=\sD_N$. 

\begin{theorem} \label{theo:main} Let $s$ be as above and suppose that $\Pi_{N,s}\neq 0$. Define 
\begin{equation*}
A_s=\sum_{i=N_0+1}^{\sD_N}(s_i-i)+\sum_{\gamma=1}^{d-1}N_\gamma S_{\gamma-1}.  
\end{equation*}
Then
\begin{equation*}
\sum_{\eta\in U} v_\eta(\Pi_{N,s})\leq A_s\big(4g+2d-4-\sum_{\xi\in \cR}(e_\xi-2)\big)-(\sD_N-N_0)\big(\sum_{\xi\in \cR}(r_\xi-e_\xi)\big)-\sum_{\gamma=1}^{d-1}N_\gamma\sum_{\xi\in \cR}v_\xi(z_\gamma).
\end{equation*}
\end{theorem}

\begin{rem}
In the important special case of $\Delta_N$, we have $s_i=i$ for all $i$, so $A_s=\sum_{\gamma=1}^{d-1}N_\gamma S_{\gamma-1}$. 
\end{rem}

\begin{proof}[Proof of Theorem~\ref{theo:main}] 
Instead of proving an upper bound on $\sum_{\eta\in U}v_\eta(\Pi_{N,s})$ we prove a lower bound on $\sum_{\xi\in \cR} v_\xi(\Pi_{N,s})$ and use the fact the degree of the divisor of a function is zero. By Lemma~\ref{lem:PiPiprime}, we can replace $\Pi_{N,s}$ by $\Pi'_{N,s}$. We want to apply Proposition~\ref{prop:valDnf} to the $\Pi'_{N,s}$. The determinant $\Pi'_{N,s}$ is an alternating sum of products of the form $\prod_{\gamma=0}^{d-1}\prod_{i=1+S_{\gamma-1}}^{S_\gamma}D_{\sigma(i)-i+S_{\gamma-1}}(z_\gamma)$, where $\sigma:\{1,2,\dots, \sD_N\}\to \{s_1,s_2,\dots,  s_{\sD_N}\}$ is a bijection. We assume from now on that this product is non-zero. Since $z_0=1$, $\cN'_{N,0}$ is the $N_0\times N_0$ identity matrix augmented by $N-N_0$ columns of zeros on the right. It follows that $\sigma(i)=i$ if $i\leq N_0$ and therefore $\sigma(i)>N_0$ for all $i>N_0$.  Thus $\sigma$ induces a bijection between $\{N_0+1,N_0+2,\dots \sD_N\}$ and $\{s_{N_0+1},s_{N_0+2}, \dots s_{\sD_N}\}$ and therefore $\sum_{i=N_0+1}^{\sD_N}\sigma(i)=\sum_{i=N_0+1}^{\sD_N}s_i$ independently of $\sigma$.  Finally, all the terms with $\gamma=0$ are equal to $1$ so we can consider $\prod_{\gamma=1}^{d-1}\prod_{i=1+S_{\gamma_{i-1}}}^{S_\gamma}D_{\sigma(i)-i+S_{\gamma-1}}(z_\gamma)$. In this product, all the terms appear in the $(N_0+1)^{\text{st}}$ column of $\cN_N$  or to the right of it.

Now the $(N_0+1)^{\text{st}}$ column of $\cN_{N,\gamma}$ is equal to the transpose of 
\begin{equation*}
(D_{N_0}(z_\gamma), D_{N_0-1}(z_\gamma), \cdots, D_{N_0-N_\gamma+1}(z_\gamma))
\end{equation*}
and as we move further to the right in $\cM_{N,\gamma}$, we take derivations of higher and higher order. On the other hand, $N_\gamma\leq \lfloor{\frac{N}{d}}\rfloor+1$ for all $\gamma$, and $N_0=\lfloor{\frac{N}{d}}\rfloor+1$, since $\cO_X(N[\infty])$ contains the $\lfloor{\frac{N}{d}}\rfloor+1$ functions $x^i$, ($0\leq i\leq \lfloor{\frac{N}{d}}\rfloor$).  Hence $N_\gamma\leq N_0$ for all $\gamma$. It follows that $N_0-N_\gamma+1\geq 1$ and therefore $\sigma(i)-i+S_{\gamma-1}\geq 1$ for all $i>N_0$. 

We therefore deduce from Proposition~\ref{prop:valDnf} that, if $\xi\in \cR$ and $\xi\neq \infty$, 
\begin{align*}
v_\xi(\prod_{\gamma=1}^{d-1}&\prod_{i=1+S_{\gamma-1}}^{S_\gamma}D_{\sigma(i)-i+S_{\gamma-1}}(z_\gamma))=\sum_{\gamma=1}^{d-1}\sum_{i=1+{S_{\gamma-1}}}^{S_\gamma}v_\xi(D_{\sigma(i)-i+S_{\gamma-1}}(z_\gamma))\\
&\geq \sum_{\gamma=1}^{d-1}\sum_{i=1+S_{\gamma-1}}^{S_\gamma}(\sigma(i)-i+S_{\gamma-1}-1)e_\xi-(2(\sigma(i)-i+S_{\gamma-1})-1)r_\xi\\
&\hskip6cm  +\sum_{\gamma=1}^{d-1}N_\gamma v_\xi(z_\gamma)\\
&=\Big(\sum_{\gamma=1}^{d-1}\sum_{i=1+S_{\gamma-1}}^{S_\gamma}(\sigma(i)-i+S_{\gamma-1})\Big)(e_\xi-2r_\xi)-(\sD_N-N_0)(e_\xi-r_\xi)\\
&\hskip6cm  +\sum_{\gamma=1}^{d-1}N_\gamma v_\xi(z_\gamma)\\
&=\Big(\sum_{i=N_0+1}^{\sD_N}(s_i-i)+\sum_{\gamma=1}^{d-1}N_\gamma S_{\gamma-1}\Big)(e_\xi-2r_\xi)-(\sD_N-N_0)(e_\xi-r_\xi)\\
&\hskip6cm  +\sum_{\gamma=1}^{d-1}N_\gamma v_\xi(z_\gamma)\\
&=A_s(e_\xi-2r_\xi)+(\sD_N-N_0)(r_\xi-e_\xi)+\sum_{\gamma=1}^{d-1}N_\gamma v_\xi(z_\gamma).
\end{align*}
Similarly, 
\begin{align*}
v_\infty(\prod_{\gamma=1}^{d-1}&\prod_{i=1+S_{\gamma-1}}^{S_\gamma}D_{\sigma(i)-i+S_{\gamma-1}}(z_\gamma))=\sum_{\gamma=1}^{d-1}\sum_{i=1+{S_{\gamma-1}}}^{S_\gamma}v_\infty(D_{\sigma(i)-i+S_{\gamma-1}}(z_\gamma))\\
&\geq \sum_{\gamma=1}^{d-1}\sum_{i=1+S_{\gamma-1}}^{S_\gamma}(3(\sigma(i)-i+S_{\gamma-1})-1)e_\infty-(2(\sigma(i)-i+S_{\gamma-1})-1)r_\infty\\
&\hskip6cm  +\sum_{\gamma=1}^{d-1}N_\gamma v_\infty(z_\gamma)\\
&=\Big(\sum_{\gamma=1}^{d-1}\sum_{i=1+S_{\gamma-1}}^{S_\gamma}(\sigma(i)-i+S_{\gamma-1})\Big)(3e_\infty-2r_\infty)-(\sD_N-N_0)(e_\infty-r_\infty)\\
&\hskip6cm  +\sum_{\gamma=1}^{d-1}N_\gamma v_\infty(z_\gamma)\\
&=A_s(3e_\infty-2r_\infty)+(\sD_N-N_0)(r_\infty-e_\infty)+\sum_{\gamma=1}^{d-1}N_\gamma v_\infty(z_\gamma).
\end{align*}

Summing over all $\xi \in \cR$ and recalling that $e_\infty=d$, we deduce that
\begin{equation*}
\sum_{\xi\in \cR}v_\xi(\Pi'_{N,s})\geq A_s\big(2d+\sum_{\xi\in \cR}(e_\xi-2r_\xi)\big)+(\sD_N-N_0)\big(\sum_{\xi\in \cR}(r_\xi-e_\xi)\big)+\sum_{\gamma=1}^{d-1}N_\gamma\sum_{\xi\in \cR}v_\xi(z_\gamma).
\end{equation*}
Finally,  writing $e_\xi-2r_\xi=(e_\xi-2)-2(r_\xi-1)$, we find
\begin{align*}
2d+\sum_{\xi\in \cR}(e_\xi-2r_\xi)&=2d-2\sum_{\xi\in \cR}(r_\xi-1)+\sum_{\xi\in \cR}(e_\xi-2)\\
&=2d-4(g+d-1)+\sum_{\xi\in \cR}(e_\xi-2)\\
&=-4g-2d+4+\sum_{\xi\in \cR}(e_\xi-2)
\end{align*}
from which the result follows.
\end{proof} 

\section{Orders of zeros at torsion points}

The purpose of this section is to prove Proposition~\ref{prop:Nplusrzeros}, and that $\Delta_N$ has a zero of order at least $N-\sD_N+1$ at any point of $X[N]$ not lying in $\cR$. In fact, these are special cases of the following more technical-looking result.

\begin{prop}\label{prop:orderzeros} Let $N\geq 2$ and let $r\geq 0$ be integers and let $\xi\in X[N]$, $\xi\notin \cR$.

\case{i} Let $s:\{1,2,\dots, \sD_N\} \to \{1,2,\dots, N\}$ be as before. Then $\Pi_{N,s}$ has a zero of order at least $N+1-s_{\sD_N}$ at $\xi$. 

\case{ii}  Let $\ell=\lfloor{\frac{r}{d}}\rfloor+1$. Then $\Delta_{N+r}$ has a zero of order at least $\ell(N+\ell-\sD_{N+r})$ at $\xi$.
\end{prop}

\begin{rem} The fact that $\Delta_N$ has a zero of order at least $N-\sD_N+1$ at a point of $X[N]-\cR$ can be viewed either as the case $s_i=i$ of \case{i} or the case $r=0$ of \case{ii} of the Proposition.  To prove Proposition~\ref{prop:Nplusrzeros},  it suffices to remark that under its hypotheses $\sD_{N+r}=N+r-g+1$, so $N+\ell-\sD_{N+r}=g-r+\lfloor{\frac{r}{d}}\rfloor$. 
\end{rem}

\begin{proof} Let $\alpha\in k(X)$ have divisor $N([\xi]-[\infty])$. We can further suppose that $\alpha=\frac{1}{t^N}+O(t^{1-N})$, and then choose $y_{\sD_N}=\alpha$. 

\case{i} We know that $D_n\alpha$ has a zero of order at least $N-n$ at $\xi$ for all $n\leq N$. Thus the entry at the intersection of the bottom row and the $j^{\text{th}}$ column of $\Pi_{N,s}$ has a zero of order at least $N+1-s_j$ at $\xi$. Since $s$ is strictly increasing, expanding $\Pi_{N,s}$ by the bottom row gives the result.

\case{ii} Write $x_0=x(\xi)$. Since $\xi\notin \cR$, $x-x_0$ is a uniformizer at $\xi$ and we can also choose $y_{\sD_{N+id}}=(x-x_0)^i\alpha$ for all $i\geq 0$.  Now $\Delta_{N+r}$ is the determinant of the matrix $\cM_{N+r}^{(N+r-\sD_{N+r})}$ or alternatively of $\cN_{N+r}^{(N+r-\sD_{N+r})}$. The rows that have been modified are all situated at the bottom of $\cN_{N+r,\gamma}^{(N+r-\sD_{N+r})}$ (where $\gamma$ is the residue class of $N \pmod{d}$) and there are $\ell$ of them. In fact the modified rows look like this:
\begin{equation*}
\begin{pmatrix}
\alpha  &             D_1\alpha                 &       \cdots   &      D_{\sD_{N+r}-1}(\alpha)  \\
(x-x_0)\alpha & D_1((x-x_0)\alpha)   & \cdots   &  D_{\sD_{N+r}-1}(x-x_0)\alpha)  \\ 
\vdots   &               \vdots                     &       \vdots   &      \vdots                  \\
(x-x_0)^{\ell-1}\alpha & D_1((x-x_0)^{\ell-1}\alpha)  &\cdots    &  D_{\sD_{N+r}-1}(x-x_0)^{\ell-1}\alpha)
\end{pmatrix}.
\end{equation*}
Each product appearing in the expansion of the determinant $\Delta_N$ contains one and only one entry of each of these rows, each entry being in a different column. Since $v_\xi(D_n(x-x_0)^i\alpha)\geq N+i-n$, the lower bound on the order of the zero at $\xi$ at each entry of this matrix can be presented as a matrix
\begin{equation*}
\begin{pmatrix}
N      &           N-1                           &   N-2                         &  \cdots   &       s            &     s-1  &   \cdots    &       s-\ell +1                     \\
N+1  &            N                             &   N-1                         &  \cdots   &      s+1          &     s     &    \cdots    &     s-\ell+2                    \\
\vdots   &               \vdots                  &     \vdots                    &   \vdots   &      \vdots        &   \vdots &     \vdots   &      \vdots                  \\
N+\ell-1&     N+\ell-2                      &   N+\ell-3                   &  \cdots    &     s+\ell-1   &    s+\ell-2   &    \cdots    &        s                            \\ 
\end{pmatrix},
\end{equation*}
where $s$ is defined as the bottom right entry, so $s=N+\ell-1-(\sD_{N+r}-1)=N+\ell-\sD_{N+r}$. Thus the smallest contribution is from products containing entries in the $\ell$ rightmost columns, and these contribute a zero of order at least $s\ell$. This completes the proof of the Proposition.
\end{proof}

\begin{rem}
In \S~\ref{sec:applex} we shall only use the case $r=0$ of \case{ii} of the Proposition. But our arguments could be modified to prove a bound for a weighted sum of cardinalities of sets $X[N-r]$ with $r$ running over a short interval starting with $r=0$. We omit the details. 
\end{rem}

\section{Non-vanishing of $\Delta_N$} \label{sec:DeltaNzero}

The main result of this section is the following Proposition. As usual, we suppose that $N\geq 2$.  If $r\in \{0,1,\dots, N-\sD_N\}$, denote by $s^{(r)}=(s_1^{(r)}, \dots, s_{\sD_N}^{(r)})$ the sequence defined by $s^{(r)}_i=i$ when $i\leq N_0$ and $s^{(r)}_i=i+N-\sD_N-r$ when $i\geq N_0+1$. 

\begin{prop} \label{prop:DeltaNzero}

\case{a} Suppose $p=0$. The for all $r\in \{0,1,\dots, N-\sD_N\}$,  $\Pi_{N,s^{(r)}}\neq 0$.

\case{b} Suppose $p>N$. If further $p$ does not divide $d$,  then $\Delta_N\neq 0$.
\end{prop}

\begin{cor}  Let $r\in \{0,1,\dots, N-\sD_N\}$. 

\case{a} Suppose $p=0$. Then $(N-r)_*X\not\subseteq W_r^{-}$. 
 
\case{b} Suppose  $p>N$. Then $(\sD_N)_*X\not\subseteq W_{N-\sD_N}^{-}$. 
\end{cor}

\begin{proof}[Proof of the Corollary]
\case{a} Note that $\Pi_{N,s^{(r)}}$ is one of the minors of $\cM_N^{(r)}$.  Therefore \case{a} follows from Corollary~\ref{cor:crucialgeneric} and \case{a} of the Proposition.

\case{b} First note that the condition $(\sD_N)_*X\not\subseteq W_{N-\sD_N}^{-}$ does not depend on the choice of $d$ made at the beginning of \S~\ref{sec:vals}. Therefore Corollary~\ref{cor:crucialgeneric} implies that the rank of $\cM_N^{(N-\sD_N)}$ is independent of this choice. By the Riemann-Roch theorem, we can choose $d$ not divisible by $p$ and then apply \case{b} of the Proposition. \end{proof}

\begin{rem} More generally, if $r\in \{0,1,\dots N-\sD_N\}$, the condition $(N-r)_*X\not\subseteq W_r^{-}$ is independent of the choice of $d$, so the rank of $\cM_N^{(r)}$ is independent of the choice of $d$. It follows that the vanishing or not of $\Delta_N$ is independent of the choice of $d$.  In particular, when $p>N$, then  $\Delta_N\neq 0$ even when $p$ divides $d$.
\end{rem}

\begin{proof}[Proof of the Proposition]
When $p>0$, the condition $p$ does not divide $d$ means that $\pi$ is tamely ramified at $\infty$.  The function $x$ defines a canonical isomorphism of the completed local ring $\hat{\cO}_{\bP^1,\infty}$ with $k[[\frac{1}{x}]]$. By Puiseux theory, the ring of formal Laurent series $k((\frac{1}{x}))$ has a unique extension of degree $d$, namely $k((\frac{1}{x^{\frac{1}{d}}}))$.  Since $\pi:X\to \bP^1$ is totally ramified of degree $d$ at $\infty$, we deduce that the fraction field of $\hat{\cO}_{X,\infty}$ is isomorphic to $k((\frac{1}{x^{\frac{1}{d}}}))$. Furthermore, since $\hat{\cO}_{X,\infty}$ is integrally closed, it must correspond under this isomorphism to $k[[\frac{1}{x^{\frac{1}{d}}}]]$. We can then view $\frac{1}{x^{\frac{1}{d}}}$ as a formal parameter on $X$ at $\infty$. Then for every $i$, $y_i$ can be viewed as an element of $k((\frac{1}{x^{\frac{1}{d}}}))$, and it has a formal expansion 
\begin{equation*}
y_i=x^{\frac{\delta_i}{d}}+O(x^{\frac{\delta_i-1}{d}})
\end{equation*} 
in ascending powers of $\frac{1}{x^{\frac{1}{d}}}$.  Thus $D_j(y_i)=\binom{\frac{\delta_i}{d}}{j}x^{-\frac{\delta_i}{d}-j}+O(x^{-\frac{\delta_i}{d}-j-1})$ for all $j\geq 0$. Now $\Pi_{N,s^{(r)}}$ is the determinant of the bottom right $(\sD_N-N_0)\times (\sD_N-N_0)$-submatrix $\tilde{\cN}_N^{(r)}$
 of $\cN_N^{(r)}$. Therefore $\Pi_{N,s^{(r)}}$ will certainly be non-zero if the determinant of its \emph{leading coefficient matrix} $L_{N,r}$ of this matrix is non-zero. By definition, this is the matrix whose $(i,j)^{\text{th}}$ entry is the coefficient of the highest power of $\frac{1}{x^{\frac{1}{d}}}$ that can possibly occur in the expansion of the $(i,j)^{\text{th}}$ entry of $\tilde{\cN}_N^{(r)}$. Therefore, up to permutation of the rows, 
\begin{equation*}
L_{N,r}=\begin{pmatrix}  \binom{\frac{\delta_{i_1}}{d}}{r'} &  \binom{\frac{\delta_{i_1}}{d}}{r'+1} & \binom{\frac{\delta_{i_1}}{d}}{r'+2} & \cdots    &   \binom{\frac{\delta_{i_1}}{d}}{N-1-r}\\
 \binom{\frac{\delta_{i_2}}{d}}{r'} &  \binom{\frac{\delta_{i_2}}{d}}{r'+1} & \binom{\frac{\delta_{i_2}}{d}}{r'+2} & \cdots    &   \binom{\frac{\delta_{i_2}}{d}}{N-1-r}\\
\vdots &  \vdots           &     \vdots                             &   \ddots    &        \vdots       \\
 \binom{\frac{\delta_{i_{\sD_N-N_0}}}{d}}{r'} &  \binom{\frac{\delta_{i_{\sD_N-N_0}}}{d}}{r'+1}& \binom{\frac{\delta_{i_{\sD_N-N_0}}}{d}}{r'+2} & \cdots    &   \binom{\frac{\delta_{i_{\sD_N-N_0}}}{d}}{N-1-r}\\
\end{pmatrix},
\end{equation*}
where $r'=N_0+N-\sD_N-r$ and $\{i_1,i_2, \dots, i_{\sD_N-N_0}\}$ is the set of $i\in \{1,2,\dots, \sD_N\}$ such that $\delta_i\not\equiv 0\pmod{d}$. 
 
The computation of the determinant of $L_{N,r}$ is just the following Lemma:

\begin{lemma} Let $x_1$, $x_2$, \dots $x_n$  be variables, let $\ell\in \bZ_+$, and let $M$ be the matrix
\begin{equation*}
M=\Big(\binom{x_i}{\ell+j-1}\Big)_{1\leq i,j\leq n}.
\end{equation*}
Then 
\begin{equation*}
\det{(M)}=\prod_{i=1}^n\binom{x_i}{\ell} \frac{\prod_{1\leq i<j\leq n}(x_j-x_i)}{\prod_{j=1}^{n-1}(\ell+j)^{n-j}.}
\end{equation*}
\end{lemma}

\begin{proof} For every $i$, $j$ write
\begin{equation*}
\binom{x_i}{\ell+j-1}=\binom{x_i}{\ell}\frac{\prod_{t=1}^{j-1}(x_i-\ell-t+1)}{\prod_{t=1}^{j-1}(\ell+t)}, 
\end{equation*}
and view $\binom{x_i}{\ell}$ as a common factor of the entries in the $i^{\text{th}}$ row and $\prod_{t=1}^{j-1}(\ell+t)^{-1}$ as a common factor of the entries in the $j^{\text{th}}$ column. This gives
\begin{equation*}
\det{(M)}=\prod_{i=1}^n\binom{x_i}{\ell}\Big( \prod_{j=1}^{n-1}(\ell+j)^{-(n-j)}\Big)\det{(N)},  
\end{equation*}
where $N=\big(\prod_{t=1}^{j-1}(x_i-\ell-t+1)\big)_{i,j}$. The $(i,j)^{\text{th}}$ entry of $N$ is a monic polynomial of degree $j-1$ in $x_i$, so column operations show that $N$ has the same determinant as the Vandermonde matrix $\big(x_i^{j-1}\big)_{i,j}$, namely $\prod_{i<j}(x_j-x_i)$.
\end{proof}

Returning to the proof of the Proposition, we deduce from the Lemma that
\begin{equation}  \label{eq:LNr}
\det{(L_{N,r})}= \prod_{t=1}^{\sD_N-N_0}\binom{\frac{\delta_{i_t}}{d}}{r'} \frac{\prod_{1\leq t<s\leq \sD_N-N_0}{(\frac{\delta_{i_s}-\delta_{i_t}}{d})}}{\prod_{j=1}^{\sD_N-N_0}(r'+j)^{\sD_N-N_0-j}}. 
\end{equation}

Suppose $p=0$. Since $d$ does not divide the $\delta_{i_t}$'s, the binomial coefficients cannot vanish. Since the $\delta_i$'s are distinct, we deduce that $\det{(L_N)}\neq 0$. This proves \case{a}. 

Suppose $p>N$ and $r=N-\sD_N$. Then $r'=N_0=\lfloor{\frac{N}{d}}\rfloor+1$. Since $0\leq \delta_{i_t}\leq N$ and $\delta_{i_t}\not\equiv 0\pmod{d}$ for all $t$, the integers $\delta_i-\ell d$ ($0\leq \ell \leq N_0-1$) all lie in the interval $[-(p-1), (p-1)]$ and are non-zero. This implies that the binomial coefficients (which are $p$-adic integers since $p$ does not divide $d$) are not divisible by $p$. Similarly, the $\delta_i$'s are all distinct $\pmod{p}$ and, since the definition of $\cN_{N}$ involves only Hasse-Schmidt derivatives of order up to $N-1$, none of the integers $r'+j$ appearing in the denominator is divisible by $p$. This proves \case{b}. 
\end{proof} 

\begin{rems}
\case{i} When $0<p\leq N$, the leading coefficient of $\Delta_N$ may or may not vanish, and when it does, $\Delta_N$ itself may or may not vanish. Examples of this were given in \cite{Bo23}, \S~5. See also subsection~\ref{sub:examples} below. 

\case{ii} Proposition~\ref{prop:DeltaNzero}~\case{a} is rather weak as stated when $r<N-\sD_N$ since from \case{b} we already know that that the rank of $\cM_N^{(N-\sD_N)}$ --- and hence that of $\cM_N^{(r)}$ when $r<N-\sD_N$ --- is maximal. But the computation leading to the formula (\ref{eq:LNr}) for the leading coefficient is valid in any characteristic and might be useful for checking the maximality of the rank of $\cM_N^{(r)}$ in cases when that of $\cM_N^{(N-\sD_N)}$ is not maximal.

\case{iii} One could compute the leading coefficient of other $\Pi_{N,s}$'s. The result will in general involve values of Schur functions (see for example Macdonald \cite{Mac95}). These polynomials are non-zero, symmetric and positive sums of monomials. Using this, one can show that in characteristic zero the leading coefficient of $\Pi_{N,s}$, and hence $\Pi_{N,s}$ itself, do not vanish. 
\end{rems} 

\section{Applications and examples}\label{sec:applex}  In this section, we discuss various applications of and examples related to Theorem~\ref{theo:main}.  As in the Introduction, we assume from now on that $d=\delta_2$. This implies that $\pi$ is separable, as one sees by decomposing $\pi$ as the composite of a purely inseparable and a separable morphism.  Thus we can apply the results of \S\S~\ref{sec:vals}--\ref{sec:DeltaNzero}.

\subsection{Torsion of small order} We begin with some simple observations concerning small values of $N$.
\begin{prop} \label{prop:smallorder} Let $N\geq 2$.
\case{a} If $N$ is a Weierstrass gap, then $X[N]=\{\infty\}$.

\case{b} Suppose $X$ contains a point of order $N$. Suppose further that $N_i=0$ for all $i\in \{1,2,\dots, d-1\}$. Then $N=d$.
\end{prop}

\begin{rems}
\case{i} If $N<d$, then $N$ is a Weierstrass gap, so we deduce that $X[N]=\{\infty\}$. 

\case{ii} Recall that by the Riemann-Roch theorem, if $N$ is a Weierstrass gap, then $N\leq 2g-1$. Similarly, the hypothesis on the $N_i$'s of \case{b} can only be satisfied if $N\leq 2g$. Thus Proposition~\ref{prop:smallorder} is only of use for small values of $N$. 

\case{iii} When $X$ is hyperelliptic and $\infty$ a Weierstrass point, we recover the fact that if $N\leq 2g$, then $X[N]=X[2]$ when $N$ is even and $X[N]=\{\infty\}$ when $N$ is odd (see \cite{Bo23}, Lemma~3.2 or Zarhin~\cite{Za19}).
\end{rems}

\begin{proof}  \case{a} Suppose $X[N]$ contains a point $\xi\neq \infty$ of order $M$ dividing $N$. Thus $M\geq 2$ and by the Abel-Jacobi theorem $k(X)$ contains a function $\beta$ with divisor $M([\xi]-[\infty])$ and hence the function $\beta^{\frac{N}{M}}$ has divisor $N([\xi]-[\infty])$, contradicting the hypothesis.  

\case{b} Since $X$ contains a point $\xi$ of order $N$, $k(X)$ contains a function $\alpha$ with divisor $N([\xi]-[\infty])$.  Now the hypothesis on the $N_i$'s implies that a $k$-basis of $\cO_X(N[\infty])$ is given by powers of $x$. Hence $\alpha$ is a polynomial in $x$, and since $v_\infty(x)=-d$ and $v_\infty(\alpha)=-N$, this polynomial must have degree $\frac{N}{d}$.  Hence $d$ divides $N$. Also, the only zero of $\alpha$ is at $\xi$ and it has order $N$, so necessarily $\alpha=c(x-x(\xi))^{\frac{N}{d}}$ for some $c\neq 0$. Hence $x-x(\xi)$ has divisor $d([\xi]-[\infty])$ and therefore $N=d$.  \end{proof}

Next we treat $d$-torsion. The following Proposition follows directly from Lemma~\ref{lem:RH} and the Abel-Jacobi theorem.

\begin{prop} \label{prop:dtorsion} We have $X[d]=\{\xi\in  \cR\mid e_\xi=d\}$. Furthermore,
\begin{equation*}
\Card{X[d]}\leq \frac{2(g+d-1)}{d-1}. 
\end{equation*}
\end{prop}

\begin{proof} Let $\xi\in X[d]$. If $\xi=\infty$, then certainly $e_\xi=d$ and $\xi\in \cR$. If $\xi\neq \infty$, then there exists $\alpha\in k(X)$ with divisor $d([\xi]-[\infty])$. Since $d=\delta_2$, $\alpha=\lambda(x-x(\xi))$ for some $\lambda\in k^\times$.  Hence $\xi\in \cR$ and $e_\xi=d$.  Conversely, let $\xi\in \cR$ be such that $e_\xi=d$. If $\xi=\infty$, then certainly $\xi\in X[d]$. If $\xi\neq \infty$, then $x-x(\xi)$ has a zero of order at least $d$ at $\xi$; since $x-x(\xi)$ has a unique pole at $\infty$ which is of order $d$ and  the divisor of a function has degree zero, this implies that the divisor of $x-x(\xi)$ is $d([\xi]-[\infty])$. Therefore $\xi\in X[d]$.

It follows that $\Card{X[d]}=\Card{\{\xi\in \cR\mid e_\xi=d\}}$.  Hence
\begin{equation*}
(d-1)\Card{X[d]}=\sum_{\xi\mid e_\xi=d}(e_\xi-1)\leq \sum_{\xi\in \cR}(e_\xi-1)\leq \sum_{\xi\in \cR}(r_\xi-1)=2(g+d-1)
\end{equation*}
and the Proposition follows. 
\end{proof}

\subsection{Simplifying the bound} We now simplify the result of Theorem~\ref{theo:main} in order to obtain a bound depending only on $N$ and $g$.  In particular, we prove Theorem~\ref{theo:intromain}. 

First note that since $\sD_N-N_0=\sum_{i=1}^{d-1}N_i$, we have $\sD_N-N_0\geq 0$.  Also, $r_\xi\geq e_\xi$ and $e_\xi\geq 2$ for all $\xi\in \cR$. Finally, since each $z_\gamma$ is finite away from $\infty$, we have, writing $\delta(\gamma)$ for the order $-v_\infty(z_\gamma)$ of the pole of $z_\gamma$ at $\infty$:
\begin{equation*}
\sum_{\eta_\in U}v_\eta(\Pi_{N,s})\leq (4g+2d-4)A_s+\sum_{\gamma=1}^{d-1}N_\gamma \delta(\gamma).
\end{equation*}

We now bound $A_s$. Again, we choose $d=\delta_2$. Then $N_0=\lfloor{\frac{N}{d}}\rfloor+1$ and, if $\gamma\geq 1$, then since there is no non-constant function in $\cO_X(X-\infty)$ with a pole of order $<d$, $N_\gamma\leq \frac{N}{d}$.  It follows that, if $\gamma\geq 1$, then
\begin{equation*}
S_{\gamma-1}\leq \Big(\frac{N}{d}+1\Big)+\sum_{\mu=1}^{\gamma-1}\frac{N}{d}=1+\frac{N}{d}\gamma.
\end{equation*}
Hence, 
\begin{equation*}
\sum_{\gamma=1}^{d-1}N_\gamma S_{\gamma-1}=\sum_{\gamma=1}^{d-1}\frac{N}{d}\Big(1+\frac{N}{d}\gamma\big)=\frac{N}{d}(d-1)+\frac{N^2}{d^2}\frac{d(d-1)}{2}\leq N+\frac{N^2}{2}. 
\end{equation*}
Therefore
\begin{equation*}
A_s\leq \sum_{i=N_0+1}^{\sD_N}(s_i-i)+ N+\frac{N^2}{2}. 
\end{equation*}

Next, we bound $\sum_{\gamma=1}^{d-1}N_\gamma \delta(\gamma)$.  Again, we bound $N_\gamma$ by $\frac{N}{d}$. By the Riemann-Roch theorem, every integer at least $g+1$ is the order of the pole at $\infty$ of some function in $\cO_X(X-\{\infty\})$, so $\delta(\gamma)\leq g+1+\gamma\leq g+d$. Hence
\begin{equation*}
\sum_{\gamma=1}^{d-1}N_\gamma \delta(\gamma)\leq (d-1)\frac{N}{d}(g+d)=N\big(1-\frac{1}{d}\big)(g+d)\leq N(g+d).
\end{equation*} 

\subsubsection{Proof of Theorem~\ref{theo:intromain}} We now specialize to the case of $\Delta_N$, so $s_i=i$ for all $i$. By Proposition~\ref{cor:crucialgeneric}, the vanishing of $\Delta_N$  is equivalent to $(\sD_N)_*X\subseteq W^-_{N-\sD_N}$. Thus we can suppose from now on that $\Delta_N\neq 0$.  For completeness, we state the generalization of Theorem~\ref{theo:intromain} to arbitrary $N$:
\begin{prop} \label{prop:genintromain} Let $N\geq 2$. If $\Delta_N\neq 0$, then
\begin{equation} \label{eq:genintromain}
\Card{(X[N])}\leq  \frac{1}{N-\sD_N+1}\big((3g-1)N^2+(8g-1)N\big)+4g.
\end{equation}
\end{prop}
Since $\sD_N=N-g+1$ when $N\geq 2g-1$, this implies Theorem~\ref{theo:intromain}.

\begin{proof}[Proof of  Proposition~\ref{prop:genintromain}]. Since $\Delta_N\neq 0$ and $s_{\sD_N}=\sD_N$, Proposition~\ref{prop:orderzeros} implies that $\Delta_N$ has a zero of order at least $N-\sD_N+1$ at any point of $X[N]-\cR$. We deduce that
\begin{align*}
\Card{(X[N]-\cR)}&\leq \frac{1}{N-\sD_N+1}\big((4g+2d-4)(N+\frac{N^2}{2})+N(g+d)\big)\\
&\leq \frac{1}{N-\sD_N+1}\big((6g-2)(N+\frac{N^2}{2})+(2g+1)N\big).
\end{align*}
Since $\rho\leq 2(g+d-1)\leq 4g$ by Lemma~\ref{lem:RH}, $(\ref{eq:genintromain})$ follows. 
\end{proof}

\subsection{The worst case} \label{subsec:worst} We begin by analysing what Theorem~\ref{theo:main} can tell us for a general choice of $X$ and $\infty$.  We keep the notation already introduced. For each residue class $\gamma\in \bZ/d\bZ$, we write $\delta(\gamma)=-v_\infty(z_\gamma)$, the order of the pole of $z_\gamma$ et $\infty$. 

\begin{prop} Let $s$ be such that $\Pi_{N,s}\neq 0$. Then
\begin{equation*}
\Card{(X[N]-\cR)}\leq \frac{1}{N+1-s_{\sD_N}}\Big(A_s(4g+2d-4)+\sum_{\gamma=1}^{d-1}N_\gamma\delta(\gamma)\Big).
\end{equation*}
\end{prop}

\begin{proof}
We apply Theorem~\ref{theo:main}. Since $s_i\geq i$ for all $i$, $A_s\geq 0$. Also, $e_\xi\geq 2$ for all $\xi\in \cR$, so $\sum_{\xi}(e_\xi-2)\geq 0$. Next, $\sD_N-N_0=\sum_{i=1}^{N-1}N_i\geq 0$ and $r_\xi\geq e_\xi$ for all $\xi\in \cR$, so $(\sD_N-N_0)\sum_{\xi}(r_\xi-e_\xi)\geq 0$. Finally, since $\infty$ si the only pole of $z_\gamma$, $\sum_{\xi}v_\xi(z_\delta)\geq -\delta(\gamma)$.  Thus Theorem~\ref{theo:main} implies that
\begin{equation*}
\sum_{\eta\in U}v_\eta(\Pi_{N,s})\leq A_s(4g+2d-4)+\sum_{\gamma=1}^{d-1}N_\gamma\delta(\gamma).
\end{equation*}
The Proposition now follows from Proposition~\ref{prop:orderzeros}.
\end{proof}

The worst bound is obtained by choosing $s=s^{(0)}$, where $s^{(0)}_i=i+N-\sD_N$ whenever $i\geq N_0+1$, since then $s_i\leq s^{(0)}_i$ for all $s$ and all $i$. Hence $\frac{1}{N+1-s_{\sD_N}}\leq \frac{1}{N+1-s^{(0)}_{\sD_N}}=1$ and $A_s\leq A_{s^{(0)}}=(N-N_0)(N-\sD_N)+\sum_{\gamma=1}^{d-1}N_\gamma S_{\gamma}$. By Corollary~\ref{cor:crucialgeneric}, at least one $\Pi_{N,s}$ is non zero, so we deduce the following:

\begin{cor} \label{cor:worst}
We always have
\begin{equation*}
\Card{(X[N]-\cR)}\leq \Big((N-N_0)(N-\sD_N)+\sum_{\gamma=1}^{d-1}N_\gamma S_{\gamma}\Big)(4g+2d-4)+\sum_{\gamma=1}^{d-1}N_\gamma\delta(\gamma). 
\end{equation*}
\end{cor}

For fixed $g$, $d$, this bound is usually much worse than that of Proposition~\ref{prop:oldbound} since it leads to a bound of $4g+2d-\frac{7}{2}$ on the coefficient of $N^2$.  However if $d$ is fixed and $N\leq g$ say, it becomes a quadratic bound. See also Remark~\ref{rem:worst} below. 

\subsection{The case of tame ramification} \begin{prop}\label{prop:tame}  Suppose that $\pi$ is tamely ramified. Then
\begin{equation*}
\sum_{\eta\in U} v_\eta(\Pi_{N,s})\leq A_s(\rho+2g-2)-\sum_{\gamma=1}^{d-1}N_\gamma\sum_{\xi\in \cR}v_\xi(z_\gamma).
\end{equation*}
\end{prop}

\begin{proof} Since $\pi$ is tamely ramified, $r_\xi=e_\xi$ for all $\xi$, so we have
\begin{equation*}
\sum_{\eta\in U} v_\eta(\Pi_{N,s})\leq A_s\big(4g+2d-4-\sum_{\xi\in \cR}(e_\xi-2)\big)- \sum_{\gamma=1}^{d-1}N_\gamma\sum_{\xi\in \cR}v_\xi(z_\gamma).
\end{equation*}
Here $\sum_{\xi\in \cR}(e_\xi-2)=-\rho+\sum_{\xi\in \cR}(e_\xi-1)=-\rho+2(g+d-1)$ and the result follows.
\end{proof}

\subsection{Hyperelliptic curves with base point a Weierstrass point} \label{subsecHCWP} We now show how Theorem~\ref{theo:main} can be used to recover some of the results of \cite{Bo23}. Suppose therefore that $X$ is hyperelliptic curve and $\infty$ a Weierstrass point.   Thus $d=2$ and, since $e_\xi\leq d$ for all $\xi\in \cR$, we in fact have $e_\xi=2$ for such $\xi$.  Also, $\sD_N=N_0+N_1$. Using Lemma~3.2 of \cite{Bo23}, we find that $N_0=\lfloor{\frac{N}{2}}\rfloor+1$, and 
\begin{equation*}
N_1=\begin{cases} 0 &\text{if $N\leq 2g$}\\
\lfloor{\frac{N+1-2g}{2}}\rfloor &\text{if $N\geq 2g+1$} 
\end{cases}
\end{equation*}
We therefore suppose $N\geq 2g+1$ from now on.  Thus
\begin{equation} \label{eq:Ashyperelliptic}
A_s=\sum_{i=N_0+1}^{N_0+N_1}(s_i-i)+N_1S_0=\sum_{i=N_0+1}^{N_0+N_1}(s_i-i)+N_0N_1
\end{equation}

\subsubsection{The case $p\neq 2$} In this case, $\pi$ is tamely ramified and we can choose a plane affine model of $X-\{\infty\}$ of the form $y^2=f(x)$, where $f$ is monic of degree $2g+1$ without multiple zeros. Then $\cR=\{\infty\}\cup \{(x_0,0)\mid f(x_0)=0\}=X[2]$, so that $\rho=2g+2$. Furthermore $y$ has a unique pole at $\infty$ of order $2g+1$ and its vanishing locus is just the set of finite points of $\cR$. We can therefore choose $z_1=y$, so by what has just been said and the fact that degree of the divisor of a function in zero, $\sum_{\xi\in \cR}v_\xi(z_1)=0$.  Therefore Proposition~\ref{prop:tame} gives
\begin{equation*}
\sum_{\eta\in U}v_\eta(\Pi_{X,s})\leq A_s(\rho+2g-2)=4gA_s
\end{equation*}
and therefore, by Proposition~\ref{prop:orderzeros},
\begin{equation*}
\Card{(X[N]-\cR)}\leq  \frac{1}{N+1-s_{\sD_N}}4gA_s.
\end{equation*}

Since $\cR=X[2]$, we see using (\ref{eq:Ashyperelliptic}) that this bound is the same (up to notation) as that of Theorem~3.8~\case{b} of \cite{Bo23}. 

\subsubsection{The case $p=2$} We can obtain the same result when $p=2$, but to do this we have to appeal directly to Theorem~\ref{theo:main}. Now $X-\{\infty\}$ has an affine model of the form $y^2+Q(x)y=P(x)$, where $P$ is monic of degree $2g+1$ and $Q\neq 0$ has degree at most $g$. Let $\pi:X\to \bP^1$ be the morphism defined by $x$. We have  $e_\xi=2$ for all $\xi\in \cR$.  We need to compute $\rho$ and the values of $r_\xi$ for all $\xi\in \cR$. This is achieved by following Lemma, which is probably well-known.

\begin{lemma} \label{lem:hyperellipticwild}
\case{i} We have $\cR=\{\infty\}\cup \{\xi=(x_0,y_0)\mid Q(x_0)=0\}$. Hence 
\begin{equation*}
\rho=1+\Card{\{x_0\in k\mid Q(x_0)=0\}}.
\end{equation*}

\case{ii} Let $\xi\in \cR$.  We have $r_\infty=2g+3-2\deg{Q}$ and, if $\xi=(x_0,y_0)\neq \infty$, then $r_\xi =2m_\xi+1$, where $m_\xi$ is the multiplicity of the zero of $Q$ at $x_0$. 
\end{lemma}

\begin{proof} \case{i} Since $\pi$ has degree $2$, finite points $\xi=(x_0,y_0)$ belong to $\cR$ if and only if $y_0$ is the \emph{unique} element of $k$ such that $y_0^2+Q(x_0)y_0=P(x_0)$. If $y_1$ is the second solution of $y^2+Q(x_0)y=P(x_0)$, then $y_0+y_1=-Q(x_0)$ so (since $p=2$), $y_1=y_0+Q(x_0)$. Thus $y_1=y_0$ if and only if $Q(x_0)=0$.

\case{ii} Its suffices to consider the case $\xi\neq \infty$. Indeed, once this is known, we find by the Riemann-Hurwitz formula that
\begin{equation*}
2g+2=\sum_{\xi\in \cR}(r_\xi-1)=r_\infty-1+\sum_{\xi\neq \infty}(r_\xi-1)=r_\infty-1+2\sum_{\xi\neq \infty}m_x=r_\infty-1+2\deg{Q},
\end{equation*}
from which the value of $r_\infty$ follows. 

So, suppose $\xi=(x_0,y_0)$. After a translation, we can suppose $x_0=y_0=0$. Then $Q(0)=0$ and also $P(0)=0$. Now $y$ is a local parameter at $\xi$, so we can write $x=\sum_{i\geq 1}A_iy^i$ in the completed local ring $\hat{\cO}_{X,\xi}$.  Write $f(x,y)=y^2+Q(x)y-P(x)$. Since $X$ is smooth and $\frac{\partial f}{\partial y}(0,0)=Q(0)=0$, we must have $\frac{\partial f}{\partial x}(0,0)=P'(0)\neq 0$.  Thus from
\begin{equation*}
y^2+(Q'(0)A_1y+O(y^2))y=P'(0)(A_1y+A_2y^2)+ D_2P(0)A_1y^2+O(y^3)
\end{equation*}
we see that $A_1=0$, then that $A_2=P'(0)^{-1}$ si non-zero. 

Write $m=m_\xi$. We have $Q(x)=x^m(D_mQ(0)+O(x))$  where $D_mQ(0)\neq 0$. Thus
\begin{equation*}
y^2+D_mQ(0)(A_2y^2+O(y^3))^my=P(x),
\end{equation*}
and the left hand side is $y^2+D_mQ(0)A_2^{2m+1}y^{2m+1}+O(y^{2m+2})$ where the coefficient of $y^{2m+1}$ is non-zero.  Since $A_2P'(0)=1$, the right hand side develops as $y^2+B_3y^3+\cdots +B_ny^n+\cdots$ where we must have $B_3=B_4=\cdots  = B_{2m}=0$ and $B_{2m+1}\neq 0$. From this we deduce successively that $A_3=0$, $A_4=0$, \dots,  $A_{2m}=0$ and $A_{2m+1}\neq 0$. 

Since $p=2$ and $2m+1$ is odd, we deduce that $r_\xi=2m+1$ as claimed.
\end{proof}

We can now apply Theorem~\ref{theo:main}. We choose $z_1=y$. Since $e_\xi=2$ for all $\xi\in \cR$, we find, using Lemma~\ref{lem:hyperellipticwild} and $\sum_{\xi\neq \infty}m_x=\deg{Q}$, that
\begin{align*}
\sum_{\eta\in U}v_\eta(\Pi_{N,s})&\leq 4gA_s-N_1\big(\sum_{\xi\in \cR}(r_\xi-e_\xi)\big)-N_1\sum_{\xi\in \cR}v_\xi(y)\\
&=4gA_s-N_1(r_\infty-2)-N_1\sum_{\xi\neq \infty}(r_\xi-2)-N_1\big(v_\infty(y)+\sum_{\xi\neq \infty}v_\xi(y)\big)\\
&=4gA_s-N_1(2g+1-2\deg{Q})-N_1\sum_{\xi\neq \infty}(2m_\xi-1)+(2g+1)N_1\\
&\hskip78mm -N_1\sum_{\xi\neq \infty}v_\xi(y)\\
&=4gA_s+N_1\big(\rho-1-\sum_{\xi\neq \infty}v_\xi(y)\big).
\end{align*}
Since $v_\infty(y)=-(2g+1)$ and $v_\infty(x)=-2$, we can choose for $z_1$ any function of the form $y+S(x)$ where $S(x)$ is a polynomial of degree at most $g$. Since $\rho-1\leq g$, we can choose $S$ so that $z_1$ vanishes at all the finite $\xi\in \cR$. Then $\sum_{\xi\neq \infty}v_\xi(y)\geq \rho-1$, so we conclude that
\begin{equation*}
\sum_{\eta\in U}v_\eta(\Pi_{N,s})\leq 4gA_s
\end{equation*}
as before.  Since $\cR=X[2]$, this leads to the same bound on $\Card{(X[N]-X[2])}$ as before. 

\begin{rem}\label{rem:worst} Applying the {\lq\lq}worst case{\rq\rq} strategy leading to Corollary~\ref{cor:worst} leads to the bound (still assuming $X$ hyperelliptic and $\infty$ a Weierstrass point)
\begin{equation*}
\Card{(X[N]-X[2])}\leq \begin{cases} g(N^2-(2g)^2)  & \text{if $N$ is even}, \\
                                                         g(N^2-(2g-1)^2) & \text{if $N$ is odd}. \end{cases}
\end{equation*}
In general, this is still weaker than Proposition~\ref{prop:oldbound}, but taking for example $N=2g+1$ gives $\Card{(X[2g+1])}\leq 8g^2$ which is better (see Proposition~3.10 of \cite{Bo23}).   
\end{rem}

\subsection{Genus two curves with base point not a Weierstrass point} \label{subsecg2notWP}  When $g=2$, the term $(3-\frac{1}{g})N^2$ of the bound in Theorem~\ref{theo:intromain} is weaker than the term $g(N+1)^2$ of Proposition~\ref{prop:oldbound}. When $\infty$ is a Weierstrass point, the results of the previous subsection improve this, so suppose that $\infty$ is not a Weierstrass point. Then by the Riemann-Roch theorem, $\delta_2=3$. Hence we  take $d=3$.  Also, Riemann-Roch implies that $\delta_n =n+1$ for all $n\geq 2$. There cannot be any points of order $2$, so suppose $N\geq 3$. Then $N_0=\lfloor{\frac{N}{3}}\rfloor+1$ and $N_i=\lfloor{\frac{N}{3}}\rfloor$ for $i=1$, $2$. Thus
\begin{equation*}
A_s\leq \sum_{i=N_0+1}^{\sD_N}(s_i-i)+\frac{N}{d}(d-1)+\frac{N^2}{d^2}\frac{d(d-1)}{2}=\sum_{i=N_0+1}^{\sD_N}(s_i-i)+\frac{2N}{3}+ \frac{N^2}{3}
\end{equation*}
(and equality holds when $3$ divides $N$). 

Next we consider the term $4g+2d-4-\sum_{\xi\in\cR}(e_\xi-2)$.  Since $e_\infty=3$, this is at most $4g+2d-4-1=9$.  Also, $\sD_N=N-1$ for all $N\geq 3$. Thus, when $\Delta_N\neq 0$, since $g=2$, we have

\begin{align*}
\Card{(X[N]-\cR)}\leq \frac{3N^2}{2}+3N&-\frac{1}{2}\Big(N-\Big\lfloor{\frac{N}{3}}\Big\rfloor-2\Big)\Big(\sum_{\xi\in \cR}(r_\xi-e_\xi)\Big)\\
&-\frac{1}{2}\Big\lfloor{\frac{N}{3}}\Big\rfloor\Big(\sum_{\xi\in \cR}(v_\xi(z_1)+v_\xi(z_2))\Big)\Big).
\end{align*}

Note that, since $\cR$ contains the unique pole $\infty$ of $z_1$ and $z_2$, $\sum_{\xi\in \cR}(v_\xi(z_1)+v_\xi(z_2))\leq 0$ so, at least in the case of tame ramification, the linear term in $N$ is always strictly positive. So asymptotically in $N$, our bound is the same as Pareschi's bound $\frac{3}{2}N^2$ \cite{Par21}, but there is an extra linear term and an extra $\rho\leq 2(g+d-1)=8$:

\begin{align*}
\Card{(X[N])}\leq \frac{3N^2}{2}+3N&-\frac{1}{2}\Big(N-\Big\lfloor{\frac{N}{3}}\Big\rfloor-2\Big)\Big(\sum_{\xi\in \cR}(r_\xi-e_\xi)\Big)\\
&-\frac{1}{2}\Big\lfloor{\frac{N}{3}}\Big\rfloor\Big(\sum_{\xi\in \cR}(v_\xi(z_1)+v_\xi(z_2))\Big)\Big)+8.
\end{align*}

\subsection{Some specific examples} \label{sub:examples} We end the paper by considering some specific examples.

\subsubsection{Examples where $\cM_N^{(r)}$ is not of maximal rank for all $r\geq 1$} \label{ex:MNrnotmax} Suppose $X$ and $N$ are such that $(1-N)_*X=X$. Then  $(N-r)_*X\subseteq  W_r^{-}$ for all $r=1$, \dots, $N-\sD_N$, so by Corollary~\ref{cor:crucialgeneric}, $\cM_N^{(r)}$ does not have maximal rank for such $r$. The converse argument also holds, so the condition that $\cM_N^{(r)}$ is not of maximal rank for all $r\geq 1$ is equivalent to $(1-N)_*X=X$.  It is well-known that this can happen only if $p>0$, $p$ is purely inseparable for $X$ and $N=p^t+1$ for some $t\geq 1$  (see for example the proof of Lemma~6.3 of \cite{Bo23}). 

Suppose for example that $X$ is defined over a finite field $\bF_{p^s}$ with $p^s$ elements and that the Frobenius endomorphism $\varphi$ of $J$ over $\bF_{p^s}$ acts by multiplication by $-p^t$ for some integer $t\geq 1$. Then by a well-known result of Weil, $s$ is even and $s=2t$. Furthermore, $-p^t_*X=X$ so we are in the above situation. Also, if $N=p^t+1$, then $J[N]=J(\bF_{p^{2t}})$, so $X[N]=X(\bF_{p^{2t}})$. Furthermore, the characteristic polynomial of $\varphi$ is $(T+p^t)^{2g}$, so
\begin{equation*}
\Card{(X[N])}=\Card{(X(\bF_{p^{2t}})}=p^{2t}+1+2g p^{t}=N^2+2(g-1)(N-1).
\end{equation*}

In this example, there are infinitely many values of $N$ for which $\Delta_N=0$. Despite this, the size of $X[N]$ only grows quadratically in $N$ and $g$. This is therefore a motivating example for the question about quadratic growth asked in the Introduction.

A concrete example is given by the curves $y^2+y=x^\ell$, where $\ell$ is a prime and $p$ a primitive root $\pmod{\ell}$ (see for example \cite{Bo23}~Proposition~5.7). Here $g=\frac{\ell-1}{2}$ and $t$ is any integer satisfying $t\equiv g\pmod{2g}$. 

By a result of Tate \cite{Ta66}, the hypothesis that $\varphi$ acts as multiplication by $-p^{t}$ or $p^t$ implies that $J$ is isogenous (over an algebraic closure of the base field) to a power of a supersingular elliptic curve. In fact, this last condition is equivalent to some power of $\varphi$ acting as a power of $p$. Thus, in the above example, $J$ is always isogenous to a power of a supersingular elliptic curve. In Example~\ref{ex:g9family}, we shall give examples where $\Delta_N=0$ and this is not the case.

\subsubsection{The curve $y^3-y=x^4$} We suppose $p\neq 2$.  The projective closure $X$ in $\bP^2$ is a smooth proper connected curve of genus $3$ with a unique point $\infty$ outside the affine curve. Since $X-\{\infty\}$ is smooth, $k[x,y]/(y^3-y-x^4)$ is integrally closed of dimension one, from which it follows that every function in $k(X)$ finite away from $\infty$ is a polynomial in $x$ and $y$.  Using this, one deduces that $\delta_2=3$, $\delta_3=4$ and $\delta_n=n+2$ when $n\geq 4$.  We study $X[N]$ for $N\in \{2,3,4,5\}$. 

Firstly $X[2]=\{\infty\}$ by Proposition~\ref{prop:smallorder} or the remark following it.

By Proposition~\ref{prop:dtorsion}, we have $\Card{(X[3])}\leq 5$. In fact, it is easy to work out $X[3]$ by applying the strategy of the proof of this Proposition: the set $\cR-\{\infty\}$ is the set of finite points $(x_0,y_0)$ where the covering $\pi:X\to \bP^1$ is ramified; since the discriminant of the polynomial $y^3-y-x_0^4$ is $-27x_0^8+4$, $\cR-\{\infty\}$ is the set of those $(x_0,y_0)$ such that $27x_0^8=4$ and $y_0$ is a multiple root of $y^3-y-x_0^4$. Thus $X[3]-\{\infty\}$ is the set of these $(x_0,y_0)$ for which $y^3-y-x_0^4=(y-y_0)^3$. From this it follows easily that in fact $X[3]=\{\infty\}$.  

Next, take $N=4$. Then $N_0=2$ and $N_1=1$, and
\begin{equation*}
\cN'_4=\begin{pmatrix}  1 & 0 & 0 & 0\\ 0 & 1 & 0 & 0 \\ y & D_1y & D_2y & D_3y \end{pmatrix}. 
\end{equation*}
Thus the finite points of $X[4]$ are the common zeros of $D_2y$ and $D_3y$ outside $\cR$.  (Recall that in general $D_ny$ has poles at the points of $\cR$ as follows from the considerations of \S~\ref{sec:vals}.) A computation shows that
\begin{equation*}
D_2y=\frac{18x^2y^2+6x^6y+6x^2}{(3y^2-1)^3}, \qquad D_3y=\frac{(-12x^9+60x)y^2-252x^5y-152x^9+4x}{(3y^2-1)^5}. 
\end{equation*}

When $p=3$, we find that $D_2y=0$ and $D_3y=x^9+x$. It follows from the previous considerations that $\cR=\{\infty\}$. Since $x^9+x=x(x^8+1)=x(x^4-2x^2+2)(x^4+2x^2+2)$ has distinct roots and for each root $x_0$ there are three values $y_0$ such that $y_0^3-y_0=x_0^4$, we deduce that $\Card{(X[4])}=28$.  Since $D_2y=0$, taking $r=N-\sD_N=1$ in Corollary~\ref{cor:crucialgeneric} implies that $(3)_*X\subseteq W_1^-$, which, since $(3)_*X$ and $W_1^-$ are both connected of dimension one, implies that $(-3)_*X=X$. Note that in this case and contrary to the previous example, no power of the absolute Frobenius endomorphism $\pi$ of $J$ acts as multiplication by $-3$. In fact, the characteristic polynomial of $\pi$ is $(T^2-3)^2(T^2+3)$, so the eigenvalues of $\pi^2$ (the only possible power) are $-3$ and $3$.  

When $p\neq 3$, a computation using the above values of $D_2y$ and $D_3y$ again shows that $\Card{(X[4])}=28$, the finite points having $x$-coordinate a root of $x(x^4-2x^2+2)(x^4+2x^2+2)$. This time, since $D_2y\neq 0$, $(-3)_*X\neq X$. 

Finally consider $N=5$. Since $5$ is a gap in the sequence $(\delta_n)_{n\geq 1}$, in fact $X[5]=\{\infty\}$ by Proposition~\ref{prop:smallorder}. Note that, since $5$ is a gap, $\sD_5=\sD_4$ and $\cN_5$ is obtained from $\cN'_4$ by adding a column on the right:
\begin{equation*}
\cN'_5=\begin{pmatrix}   1 &  0  &  0  &  0  &  0\\  0 & 1 & 0 & 0 & 0\\ y & D_1y & D_2y & D_3y & D_4y \end{pmatrix}.  
\end{equation*}
Since $X[5]=\{\infty\}$, $\cN_5(\xi)$ has rank $3$ at all $\xi\in X-\cR$. In particular, by Corollary~\ref{cor:crucial} and the preceeding discussion, $D_4y$ cannot vanish at the points of $X[4]$, so $D_4y\neq 0$.  Although this is somewhat artificial, it does show that our theory sometimes enables one to show that  $\Pi_{N,s}$ is non-zero without actually computing it.

\subsubsection{The curves $y^3=x+ax^5+x^{10}$ in characteristic $5$} \label{ex:g9family} Finally we give an example of a non-constant family of curves (over the affine line) where $\Delta_N=0$. We suppose $p=5$.  Again, since the affine curve is smooth, the integral domain $k[x,y]/(y^3-x-ax^5-x^{10})$ is integrally closed of dimension one. The na\"{\i}ve projective closure $\tilde{X}_a\subseteq \bP^2$ with equation $y^3z^7=xz^9+ax^5z^5+x^{10}$ has a unique point $\tilde{\infty}=(0:1:0)$ at $\infty$ which, since $3$ and $10$ are coprime, is a cusp. Therefore the normalization $X_a\to \tilde{X_a}$ has a unique point $\infty$ above $\tilde{\infty}$ and $X_a-\{\infty\} \to \tilde{X}_a-\{\tilde{\infty}\}$ is an isomorphism. Furthermore, $x$ has a pole of order $3$ at $\infty$ and $y$ a pole of order $10$ and these functions are finite elsewhere. Again, every function finite away from $\infty$ is a polynomial in $x$ and $y$. Hence the genus, which is the number of gaps in the sequence $(\delta_n)_{n\geq 1}$ is equal to $9$.  Since $x$ has a pole of order $3$, we take $d=3$. Also, $\cR=\{\infty\}\cup \{(x_0,0)\mid x_0+ax_0^5+x_0^{10}=0\}=X[3]$.

To show that the family is not constant, we shall show that the curves $X_0$ and $X_1$ are not isomorphic, and in fact that their Jacobians $J_0$ and $J_1$ (considered simply as abelian varieties) are not isogenous over an algebraic closure of $\bF_5$. 

We first consider $X_0$, which has equation $y^3=x+x^{10}$. If $u^3=-x$ and $v^{9}=-\frac{y}{u}$, then $\bF_5(u,v)\subseteq \bF_5(x,y)$ and $u^{27}+v^{27}=1$, so there is a surjective morphism $\cF_{27}\to X_0$, where $\cF_{27}$ is the (projective) Fermat curve of degree $27$. Note that $5$ generates the multiplicative group $(\bZ/27\bZ)^\times$. Recall that in general, if $m\in \bZ_{\geq 1}$ is prime to $p$ and $p$ generates the multiplicative group $(\bZ/m\bZ)^\times$, then the roots of the characteristic polynomial of the Frobenius endomorphism of the degree $m$ Fermat curve $\cF_m$ over $\bF_p$ are of the form $\zeta\sqrt{p}$ for some root of unity $\zeta$. This follows from classical formulae for the number of points of $\cF_m$ over finite extensions of $\bF_p$ (see for example \cite{Lang78}, Chapter 1) and the fact that $p$ generates a prime ideal in the ring of integers of the $m^{\text{th}}$ cyclotomic field. By \cite{Ta66}, we deduce that the jacobian of $\cF_{27}$, and hence also its quotient $J_0$, is isogenous to a power of a supersingular elliptic curve.

Now we consider $X_1$. A computation shows that $\Card{X_1(\bF_{5^4})}=696=5^4+1+70$. If $J_1$ were isogenous to a power of a supersingular curve, then again by \cite{Ta66}, the roots of the characteristic polynomial of the Frobenius endomorphism over $\bF_{5^4}$ would be of the form $5^2\zeta$ for some root of unity $\zeta$, so their sum would be an integer divisible by $5^2$.  But the computation shows that their sum is $-70$, so this is not the case and $J_1$ is not isogenous to $J_0$. 

Returning to $X_a$ with $a$ general, we first consider $N=10$. Then $\sD_N=5$, $N_0=4$, $N_1=1$ and $N_2=0$, so
\begin{equation*}
\cN'_{10}=\begin{pmatrix}  1 & 0 & 0 & 0 & 0 & 0 & 0 & 0 & 0 & 0\\ 0 & 1 & 0 & 0 & 0 & 0 & 0 & 0 & 0 & 0\\ 0 & 0 & 1 & 0 & 0 & 0 & 0 & 0 & 0 & 0\\ 0 & 0 & 0 & 1 & 0 & 0 & 0 & 0 & 0 & 0\\ y& D_1y& D_2y & D_3y & D_4y & D_5y & D_6y & D_7y & D_8y& D_9y \end{pmatrix}. 
\end{equation*} 
Thus $\Delta_{10}=D_4y$. But $D_4y=0$, so $\Delta_{10}=0$.  From Corollary~\ref{cor:crucialgeneric}, we deduce that $(5)_*X\subseteq W_5^-$.   Next, a computation shows that the $D_ny$'s, where $4\leq n\leq 9$, do not simultaneously vanish outside $\cR$, so $X[10]\subseteq \cR$. Since $\cR=X[3]$, this implies that $X[10]=\{\infty\}$.  Thus is in contrast with Example~\ref{ex:MNrnotmax} where we also had $\Delta_N=0$, but in which $X[N]$ is large. Also $D_5y=0$ but $D_6y\neq 0$, so $\cM_{10}^{(r)}$ has maximal rank when $r\in \{0,1,2,3\}$. In particular $(-9)_*X\neq X$ which is again in contrast with Example~\ref{ex:MNrnotmax}.  

Since $11$ is a Weierstrass gap, $\cN'_{11}$ is obtained from $\cN'_{10}$ by adding a column, transpose of $(0,0,0,0,D_{10}y)$, on the right. Thus again $\Delta_{11}=D_4y=0$ and we deduce that $(5)_*X\subseteq W_6^-$ which is weaker than what we obtained from $N=10$. But now because of the extra column, we deduce that $(9)_*X\not\subseteq W_2^-$, which is stronger than the information obtained from $N=10$. Also $X[11]=\{\infty\}$, either by Proposition~\ref{prop:smallorder} or by a similar argument to the case $N=10$. 

Similarly $\Delta_{12}=D_5y=0$ and $\Delta_{13}=\Delta_{14}=(D_5y)^2-D_4y\,D_6y=0$.  On the other hand, $\Delta_{15}=(D_6y)^2-D_5y\,D_7y=(D_6y)^2\neq 0$.  It is an open question whether there are infinitely many values of $N$ for which $\Delta_N=0$.

\end{document}